\documentclass[a4paper]{article}

\usepackage[left=3cm, right=3cm, top=3cm, bottom=3cm]{geometry}

\usepackage[noadjust]{cite}
\usepackage{amsmath,amssymb,amsfonts,amsthm}
\usepackage{algorithmic}
\usepackage{graphicx}
\usepackage{textcomp}
\usepackage{stmaryrd}
\usepackage{mathtools}
\usepackage{xcolor}
\usepackage{mathtools}
\usepackage{hyperref}
\usepackage{mathrsfs}
\usepackage{bbm}
\usepackage{stmaryrd}
\usepackage{pgf}
\usepackage{xcolor}
\usepackage{mathtools}
\usepackage{hyperref}
\usepackage{bbold}
\usepackage{changepage}
\usepackage{csquotes}
\usepackage{array}
\usepackage{enumitem}
\usepackage[normalem]{ulem}

\usepackage{tikz}
\usetikzlibrary{arrows.meta}

\newcommand{\N}{\mathbb{N}}
\newcommand{\C}{\mathbb{C}}

\newcommand{\R}{\mathbb{R}}

\newtheorem{theorem}{Theorem}[section]
\newtheorem{lemma}[theorem]{Lemma}
\newtheorem{corollary}[theorem]{Corollary}
\newtheorem{proposition}[theorem]{Proposition}

\newtheorem{conjecture}[theorem]{Conjecture}
\theoremstyle{definition}
\newtheorem{definition}[theorem]{Definition}
\newtheorem{remark}[theorem]{Remark}
\newtheorem{example}[theorem]{Example}
\renewcommand{\Re}{\operatorname{Re}}
\renewcommand{\Im}{\operatorname{Im}}

\DeclareMathOperator{\Var}{Var}
\DeclareMathOperator{\I}{Id}
\DeclareMathOperator{\id}{id}

\renewcommand{\Re}{\operatorname{Re}}
\renewcommand{\Im}{\operatorname{Im}}
\DeclareMathOperator{\cl}{cl}
\DeclarePairedDelimiter{\norm}{\lVert}{\rVert}
\DeclarePairedDelimiter{\abs}{\lvert}{\rvert}

\numberwithin{figure}{section}

\begin{document}

\makeatletter
{\def\@thefnmark{}\def\hyper@@anchor{}\@footnotetext{This work was supported by the French National Research Agency (ANR) as part of the ``Investments for the Future'' Program, through the ``ADI 2023'' Project funded by IDEX Paris-Saclay under Grant ANR-11-IDEX-003.}}
\makeatother

\title{Strong Stability of Linear Functional Equations with Distributed Delays}
\author{Yacine Chitour\thanks{Université Paris-Saclay, CNRS, CentraleSupélec, Laboratoire des Signaux et Systèmes, 91190 Gif-sur-Yvette, France.} \thanks{Fédération de Mathématiques de CentraleSupélec, 91190 Gif-sur-Yvette, France.} \and Felipe Gonçalves Netto\thanks{Université Paris-Saclay, CNRS, CentraleSupélec, Inria, Laboratoire des Signaux et Systèmes, 91190 Gif-sur-Yvette, France.} \thanks{Departamento de Matemática, Universidade de Brasília, Brasília--DF 70910-900, Brazil.} \and Guilherme Mazanti\footnotemark[3] \footnotemark[2]}
\date{}
\maketitle

\abstract{This paper considers linear functional equations on $\mathbb R^d$ with distributed delays defined by matrix-valued measures of bounded variation. More precisely, we are interested in providing conditions to ensure that the exponential stability of these systems is preserved under small changes of the parameters which define them. In the special case of difference equations, it is known that exponential stability is preserved under small perturbations of the matrices defining the system, but not of the delays, and an additional condition for preservation of exponential stability under perturbation of the delays is given by the Hale--Silkowski criterion (HSC). In this paper, we extend the treatment of these issues to more general systems. For that purpose, we first put forward an appropriate definition of perturbation on the delays and then propose a conjecture in the spirit of (HSC). We prove several partial results related to that conjecture.

\bigskip


}




\tableofcontents

\section{Introduction}

In this paper, we consider the linear functional equation with distributed delays
\begin{equation}\label{sys:general}
\Sigma_{M}:\quad x(t)=\int_{-1}^0 d M(\theta)x(t+\theta), 
\quad t\geq 0,
\end{equation}
where the unknown function $x$ takes values in $\R^d$ for a given positive integer $d$ and $M$ is a matrix-valued function defined on $[-1, 0]$ with entries having bounded variation.

System~\eqref{sys:general} can be seen as a particular type of time-delay system, as the evolution of the state $x$ after a given time $t$ depends not only on the value of $x(t)$, but also on past values of $x$, namely those on $[t-1, t]$. Time-delay systems have been extensively studied in the literature, and we refer the interested reader to \cite{Hale1993Introduction, Michiels2014Stability, Diekmann1995Delay, Bellman1963Differential, Hale1977Theory} for general expositions on this class of systems. The general form of \eqref{sys:general}, with a matrix-valued function $M$, allows one to represent several phenomena, including pointwise delays, corresponding to discontinuity points of $M$, as well as a distributed delay, corresponding to the absolutely continuous part of $M$.

Applications for linear functional equations of the form \eqref{sys:general} include both the analysis of neutral delay equations and the study of hyperbolic partial differential equations (PDEs). Indeed, it has been shown (see, e.g., \cite{Henry1974Linear, Hale2002Strong, Hale1993Introduction}) that stability properties of linear neutral functional differential equations are related to similar properties for linear functional equations of the form \eqref{sys:general} and, in particular, the exponential stability of \eqref{sys:general} is a necessary condition for the exponential stability of some more general neutral systems. As for hyperbolic PDEs, it is well known since at least the 1960s that certain one-dimensional problems can be reformulated as functional equations of the form \eqref{sys:general} (see, e.g., \cite{Auriol2019Explicit, Chitour2016Stability, Cooke1968Differential, Baratchart2021Sufficient, Bastin2016Stability, Slemrod1971Nonexistence, Miranker1961Periodic, Brayton1966Bifurcation, Brayton1967Nonlinear, Greenberg1984Effect, Chitour2021One}). This connection allows properties of the PDE to be investigated via the corresponding delay system. Moreover, when the PDE involves in-domain coupling, the resulting representation may naturally incorporate both discrete and distributed delays, with the latter expressed through integral terms over past states, as detailed in \cite{Auriol2019Explicit}.

In many practical applications, systems are subject to uncertainties arising from modeling inaccuracies or physical limitations in measurement and implementation. In this context, small variations in the parameters can significantly affect the system dynamics. Therefore, it is essential to investigate how the system behavior changes under such variations, in order to ensure robustness, stability, and satisfactory performance even under non-ideal conditions. 

To provide a more precise approach to the problem described previously in the context of robustness of stability, we must consider the asymptotic behavior of trajectories of System~\eqref{sys:general}. To define the latter, one first needs to define a state space and then ensure existence of such trajectories for all nonnegative times. In this paper, the state space is the Banach space of continuous functions defined on $[-1, 0]$ equipped with the supremum norm and satisfying a suitable compatibility condition (see \eqref{eq:compatibility} below). Then, under a mild condition on $M$ (see Proposition~\ref{prop:wellposed} and \eqref{eq:exp-estimate}), we obtain existence of solutions for all nonnegative times and also a bound for trajectories with an exponential rate independent of a particular trajectory. The least of these exponential rates is referred to as the \emph{exponential growth} associated with System~\eqref{sys:general} (see Definition~\ref{def:big-def}). The above mentioned robustness issue can now be rephrased as the continuity (in a suitable sense) of the exponential growth with respect to $M$.

These issues have been systematically addressed for the 
class of linear difference equations (cf.\ \cite{Hale1993Introduction, Avellar1980Zeros, Henry1974Linear, Hale2002Strong, Hale2003Stability, Michiels2009Strong, Chitour2016Stability, Melvin1974Stability, Cruz1970Stability, Silkowski1976Star, Avellar1990Difference}), which consists of the particular instance of System~\eqref{sys:general} defined by
\begin{equation}\label{sys:difference}
\Sigma(\tau)\colon \qquad 
x(t)=\sum_{k=1}^N A_k x(t-\tau_k),
\end{equation}
where $x(t)\in \mathbb R^d$ and, for $k \in \{1, \dotsc, N\}$, $A_k$ is a $d\times d$ matrix and $\tau_k>0$ is a delay. Note that, assuming $\tau_k \in (0, 1)$ for all $k \in \{1, \dotsc, N\}$, \eqref{sys:difference} corresponds to \eqref{sys:general} with $M$ being a suitable sum of indicator functions, namely
\begin{equation}
\label{eq:M-finite-sum}
M = \sum_{k = 1}^N A_k \mathbbm 1_{[-\tau_k, 0]}.
\end{equation}

Variations of the parameters of \eqref{sys:difference} correspond to variations in the matrices $A_1, \dotsc, A_N$ and in the delays $\tau_1, \dotsc, \tau_N$. It can be shown that, if \eqref{sys:difference} is exponentially stable, then, for small changes in the matrices, the system remains exponentially stable \cite[Section~9.6]{Hale1993Introduction}. However, the situation is different with respect to perturbation in delays, as illustrated in \cite{Silkowski1976Star} (see also \cite{Avellar1980Zeros} or \cite[Section~9.6]{Hale1993Introduction}), which provides an example in dimension $1$ showing that, for arbitrarily small changes in the delays, the system may lose the property of exponential stability. Thus, in order to preserve exponential stability under small variations in the delays (a property known as \emph{local strong stability}), one must consider additional hypotheses on the matrices $A_1, \dotsc, A_N$ and the delays $\tau_1, \dotsc, \tau_N$ than the mere exponential stability of the corresponding system \eqref{sys:difference}. The Hale--Silkowski criterion (see \cite{Silkowski1976Star, Hale1993Introduction, Avellar1980Zeros} and also Theorem~\ref{HS-Theorem} below) provides a complete answer to that issue. It establishes that, if
\begin{equation}\label{eq:HScriterion}
\max_{(\theta_1, \dotsc,\theta_N)\in [0,2\pi]^N}
\rho\!\left(\sum_{k=1}^N A_k e^{i\theta_k}\right) < 1,
\end{equation}
where $\rho(\cdot)$ denotes the spectral radius, then the system preserves exponential stability not only under small perturbations in the delays, but, surprisingly, under any perturbation in the delays (a property known as \emph{strong stability}). This criterion goes further and also asserts that condition \eqref{eq:HScriterion} is equivalent to the exponential stability of System~\eqref{sys:difference} under (small or arbitrary) delay perturbations, and to its exponential stability for one single set of rationally independent delays. As a consequence, note that robustness of exponential stability of System~\eqref{sys:difference} with respect to delay perturbations is a sole property of the matrices $A_k$.

The main goal of the present paper consists in extending these robustness results, and in particular the Hale--Silkowski criterion, to the class of systems defined in \eqref{sys:general}. For that purpose, the first step is to properly define what should be considered as a perturbation of $M$ in \eqref{sys:general}. Note that, for the particular case of \eqref{sys:difference} in which $M$ is given by \eqref{eq:M-finite-sum}, the matrices $A_1, \dotsc, A_N$ and the delays $\tau_1, \dotsc, \tau_N$ play different roles. In particular, a perturbation of $A_1, \dotsc, A_N$ yields a perturbation of $M$ in the total variation norm, but this is not the case for a variation of the delays $\tau_1, \dotsc, \tau_N$ (see Example~\ref{expl:var-delay} below). 

The first main result we prove in this paper concerns perturbations of $M$ in total variation norm, and can be seen as the counterpart of the fact that exponential stability of \eqref{sys:difference} is preserved under perturbations of the matrices $A_1, \dotsc, A_N$. In particular, we prove that, if $M$ is such that \eqref{sys:general} is exponentially stable, then \eqref{sys:general} will remain exponentially stable for all other matrix-valued functions close enough to $M$ in total variation norm (see Corollary~\ref{coro:tvstab} below), at least for an important class of functions $M$, those without a singular part.

In order to extend the Hale--Silkowski criterion to \eqref{sys:general}, one needs to generalize the notion of perturbation in delays for \eqref{sys:difference} to \eqref{sys:general}. Note that, in \eqref{sys:general}, the delay comes from the fact that the unknown function $x$ is evaluated at $t + \theta$ in the right-hand side of the equation, for $\theta \in [-1, 0]$. The variable $\theta$ can be seen as a ``distributed delay'' in \eqref{sys:general}, 
and hence a natural candidate for being a ``perturbation of the delays'' of System~\eqref{sys:general} 
consists in replacing $\theta$ by $\varphi(\theta)$, where $\varphi\colon [-1, 0] \to [-1, 0]$ is a perturbation of the identity map. System~\eqref{sys:general} is then replaced by the perturbed system given by
\begin{equation}
\label{sys:perturbed-intro}
x(t) = \int_{-1}^0 d M(\theta) x(t + \varphi(\theta)).
\end{equation}
Note that, in the case of System~\eqref{sys:difference}, System~\eqref{sys:perturbed-intro} corresponds exactly to perturbing the delays $\tau_1, \dotsc, \tau_N$ to $-\varphi(-\tau_1), \dotsc, -\varphi(-\tau_N)$. System~\eqref{sys:perturbed-intro} can also be seen as an instance of System~\eqref{sys:general} where $M$ is replaced by the pushforward by $\varphi$ of the measure associated with $M$ (see Section~\ref{sec:main-strong-stab} for more details).

The question we are interested in is thus to provide a necessary and sufficient condition on $M$ for the exponential stability of \eqref{sys:perturbed-intro} for any $\varphi$ close enough to the identity map. To proceed, the first step consists in modifying the definitions of (local) strong stability in accordance with the new concept of perturbation in the delays provided in \eqref{sys:perturbed-intro}. Then, in a second step, we generalize \eqref{eq:HScriterion} by replacing its left-hand side by the nonnegative number $\rho_{\mathrm{HS}}(M)$ given by
\begin{equation}
\label{eq:rhoHS-intro}
\rho_{\mathrm{HS}}(M) = \sup_{\xi}  \rho \left(\int_{-1}^0 e^{i\xi(\theta)}dM(\theta)\right),
\end{equation}
where the supremum is taken over Borel-measurable functions $\xi\colon [-1, 0] \to \mathbb R$. Note that the integral in the above expression reduces to the sum in \eqref{eq:HScriterion} when $M$ is given by \eqref{eq:M-finite-sum}, with $\theta_k = \xi(-\tau_k)$.

We can now state the following conjecture: $\rho_{\mathrm{HS}}(M) < 1$ if and only if strong stability of System~\eqref{sys:general} (local or not) holds true. This conjecture was shown to be true in the scalar case in \cite{GoncalvesNetto2025Strong}, in which case $\rho_{\mathrm{HS}}(M)$ reduces to the total variation of $M$. In this paper, we provide a partial result towards the proof of this conjecture in any dimension $d$, by showing (see Theorem~\ref{GHS-Theorem}) that $\rho_{\mathrm{HS}}(M) < 1$ is a necessary and sufficient condition for strong stability (local or not) of System~\eqref{sys:general} \emph{with uniform rate}, a slightly stronger notion (see Definition~\ref{def:stron-stab-UR}). As a consequence, we also obtain that $\rho_{\mathrm{HS}}(M) < 1$ is a sufficient condition for (local) strong stability, and that the latter, in turns, implies $\rho_{\mathrm{HS}}(M) \leq 1$ (see the diagram in Figure~\ref{fig:results}). As our strategy relies partly on spectral methods, we restrict our attention only to functions $M$ which satisfy the spectrum-determined growth condition (see Definition~\ref{def:big-def} below), a condition that is satisfied in particular for functions $M$ without singular part.

The paper is organized as follows. Section~\ref{sec:preliminaries} introduces the main notations used in the paper, recalls a well-posedness result for System~\eqref{sys:general}, and provides the main definitions used for the sequel. Section~\ref{sec:tv-norm} is interested in the problem of robustness of the exponential stability of System~\eqref{sys:general} with respect to perturbations of $M$ in total variation norm. In Section~\ref{sec:strong-stability}, we consider the perturbations of System~\eqref{sys:general} given by \eqref{sys:perturbed-intro}. We first recall the classical Hale--Silkowski criterion for System~\eqref{sys:difference}, before providing the definition (local) strong stability for System~\eqref{sys:general} and stating and proving our main result on this problem, Theorem~\ref{GHS-Theorem}. Appendix~\ref{sec:appendix} completes the paper with some useful measure-theoretical or technical results.

\section{Notations, definitions, and basic results}\label{sec:preliminaries}

\subsection{Elementary notations}

Let us provide some notation used throughout this paper. The set of nonnegative and positive real numbers are denoted by $\mathbb R_+$ and $\mathbb R_+^\ast$, respectively. We denote by $\mathcal M_d(\mathbb R)$ the set of all $d \times d$ real matrices, and by $\I$ the identity matrix in $\mathcal M_d(\mathbb R)$. The symbol $\lvert \cdot \rvert$ stands for an arbitrary norm on $\mathbb R^d$ (also used for the absolute value of numbers in $\R$ and $\C$), while $\lVert \cdot \rVert$ denotes the corresponding induced matrix norm on $\mathcal M_d(\mathbb R)$. For a function $f$ defined on a subset of the real line, the one-sided limits at a point $a$, when they exist, are written as $f(a^+)$ (right limit) and $f(a^-)$ (left limit). The indicator function of a set $A$ is denoted by $\mathbbm{1}_A$, and the identity map on a given set is denoted by $\id$. The space of continuous functions taking values in $\R^d$ and defined on the interval $[a,b]$ is denoted by $\mathcal C([a, b], \R^d)$, and $\lVert \cdot \rVert_\infty$ stands for the supremum norm on this space. Finally we use $\mathcal{B}_{[a,b]}$ for the Borel sets contained in $[a,b]$.

Although we consider in this paper real solutions of \eqref{sys:general} only, due to our use of spectral methods, we also need to deal with complex vectors and matrices. Given a norm $\abs{\cdot}_{\C^d}$ on $\C^d$ and its restriction $\abs{\cdot}_{\R^d}$ on $\R^d$, it is easy to find examples so that the respective induced matrix norms on $\mathcal{M}_d(\C)$ and $\mathcal M_d (\R)$ do not coincide for real matrices. For this reason, given an arbitrary norm $\abs{\cdot}_{\R^d}$ in $\R^d$, we extend this norm to the space $\C^d$ by setting, for $x \in \C^d$,
\[\abs{x}_{\C^d}=\max_{z\in \C, \abs{z}\leq 1} \abs{\Re(zx) }_{\R^d},\]
where the real part of a complex vector should be understood as the componentwise real part. This is indeed an extension because $\abs{\Re(zx)}_{\R^d}=\abs{\Re(z)}\abs{x}_{\R^d}$ for $x\in \R^d$. Consider $\norm{\cdot}_{\mathcal M_d(\R)}$ and $\norm{\cdot}_{\mathcal M_d(\C)}$, the induced matrix norms by the vector norms $\abs{\cdot}_{\R^d}$ and $\abs{\cdot}_{\C^d}$, respectively. Then one can check that $\norm{A}_{\mathcal M_d(\C)}=\norm{A}_{\mathcal M_d(\R)}$ for every $A \in \mathcal M_d (\R)$. From now on, we use $\abs{\cdot}$ to denote both $\abs{\cdot}_{\R^d}$ and $\abs{\cdot}_{\C^d}$ and we use $\norm{\cdot}$ to denote both $\norm{\cdot}_{\mathcal M_d(\R)}$ and $\norm{\cdot}_{\mathcal M_d(\C)}$.

\subsection{Main system and its well-posedness}

In this paper, we consider systems defined in \eqref{sys:general} where the entries of $M$ have bounded variation and are 
assumed to be continuous from the right on $(-1,0)$, and the integral in the right-hand side of \eqref{sys:general} is taken in the Lebesgue--Stieltjes sense. Note that the lower bound of integration in \eqref{sys:general} being equal to $-1$ is just a normalization convention, and the results of this paper also apply to systems where $-1$ is replaced by $-r$ for some $r > 0$ by a standard time rescaling.

Let us start by providing the definition of solution of System~\eqref{sys:general} that we will use in this paper.

\begin{definition}
A function $x$ is said to be a \emph{local solution} of \eqref{sys:general} if there is $\delta > 0$ such that $x\in \mathcal C([-1,\delta), \R^d)$ and $x$ satisfies \eqref{sys:general} for all $t \in [0, \delta)$. We say that $x$ is a \emph{global solution} of \eqref{sys:general} if it satisfies the definition of local solution with $\delta = +\infty$. In both cases, the restriction of $x$ to the interval $[-1, 0]$ is called the \emph{initial condition} of $x$.
\end{definition}

As usual with time-delay systems, given a function $x \in \mathcal C([-1, \delta), \R^d)$ for some $\delta \in \R_+^\ast \cup \{+\infty\}$ and $t \in [0, \delta)$, the \emph{history function} of $x$ at time $t$ is the function $x_t \in \mathcal C([-1, 0], \R^d)$ defined by
\[
x_t(\theta) = x(t+\theta), \quad \theta \in [-1, 0].
\]
With this notation, the initial condition of a solution $x$ of \eqref{sys:general} is simply $x_0$. 


Note that, if $x$ is a solution of \eqref{sys:general}, then the initial condition $x_0$ necessarily satisfies
\begin{equation}
\label{eq:compatibility}
x_0(0) = \int_{-1}^0 d M(\theta) x_0(\theta).
\end{equation}
We will thus consider in the sequel only initial conditions in the set $\mathcal C_0$ defined by
\[
\mathcal C_0 = \left\{\eta \in \mathcal C([-1, 0], \R^d) : \eta(0) = \int_{-1}^0 d M(\theta) \eta(\theta)\right\}.
\]
Moreover, if $x$
is a solution of \eqref{sys:general}, then $x_t \in \mathcal C_0$ for every $t \geq 0$ in the interval of definition of $x$.

System~\eqref{sys:general} involves the integral of a continuous function $x$ with respect to a matrix-valued function $M$, which can be understood as in the usual Riemann--Stieltjes sense, but also as the integral with respect to a matrix-valued measure associated with $M$, in the usual Lebesgue--Stieltjes sense. Later, we will deal with functions that may not be integrable in the Riemann--Stieltjes sense, making it necessary to consider integrals in the Lebesgue--Stieltjes sense. The use of these two notions of integration is consistent, since if a function is integrable in the Riemann–Stieltjes sense, then it is also integrable in the Lebesgue–Stieltjes sense, and the two integrals coincide. As both such interpretations are useful for our analysis, we next detail these points of view, starting by providing the definitions of total variation of a matrix-valued function and of a matrix-valued Borel measure.

\begin{definition}
\label{def:variations}
Given a matrix-valued function $M\colon [a,b] \to \mathcal M_d (\R)$ and a matrix-valued Borel measure $\mu\colon \mathcal{B}_{[a,b]}\to \mathcal M_d (\R)$ (i.e., each entry is a Borel signed measure), we define the \emph{total variation} of $M$ and $\mu$ in $[a, b]$ as
\[\Var (M)_{\rvert[a,b]}=\sup\left\{\sum_{i=1}^k \norm{M(t_{i+1})-M(t_{i})}: k\in \N \text{ and }a=t_1<t_2<\dotsb<t_{k+1}=b \right\} \]
and
\[\norm{\mu}([a,b])=\sup\left\{\sum_{i=1}^k \norm{\mu(E_i)}: k\in \N, E_1, \dotsc,E_{k}\in \mathcal{B}_{[a,b]} \text{ are disjoint and} \bigcup_{j=1}^{k}E_j =[a,b]\right\}, \]
respectively.
We say that $M$ is of \emph{bounded variation} in $[a, b]$ if $\Var (M)_{\rvert[a, b]} < +\infty$.
\end{definition}

\begin{definition}
A matrix-valued function of bounded variation $M\colon [a,b]\to \mathcal M_d (\R)$ is called \emph{normalized} if $M(a)=0$ and $M$ is continuous from the right on $(a,b)$. The space of normalized bounded variation functions is denoted by $\mathrm{NBV}([a,b],\mathcal M_d (\R))$.
\end{definition}

\begin{remark}\label{remk:M-mu}
There is a one-to-one correspondence (see e.g.\ \cite[Appendix~I]{Diekmann1995Delay}) between normalized bounded variation functions $M\colon [a,b]\to \mathcal M_d(\R)$ and matrix-valued Borel measures $\mu_M$ with finite total variation, expressed by
\[M(t)=\mu_M([a,t]), \quad t \in (a, b].\]
In addition, as shown in Proposition~\ref{proposition:tvcoincides} in Appendix~\ref{app:measure}, we have 
$\Var(M)_{\rvert [a, b]} = \norm{\mu_M}([a, b])$. Note also that, from the results in \cite[Appendix~I]{Diekmann1995Delay}, thanks to the normalization $M(a)=0$ for $M\in \mathrm{NBV}([a,b],\mathcal M_d (\R))$, we have that $\norm{\cdot}_{TV}:=\Var (\cdot)_{\rvert[a, b]}$ is a norm in the space $\mathrm{NBV}([a,b],\mathcal M_d (\R))$, which renders it a Banach space.

We also remark that, in the context of \eqref{sys:general}, the normalization of $M$ is not a restriction, since the integral in \eqref{sys:general} does not change if a constant is added to $M$ or if $M$ is modified in all its discontinuity points in $(-1, 0)$ in order to become right-continuous.
\end{remark}

\begin{remark}
Given a matrix-valued Borel measure of finite total variation $\mu \colon \mathcal{B}_{[a,b]}\to \mathcal M_d(\R)$, applying \cite[Theorem~19.61]{Hewitt1975Real} componentwise, we can decompose  $\mu$ as 
    \[\mu=\mu_{ac}+\mu_{sc}+\mu_d,\]
    where $\mu_{ac}$ is absolute continuous with respect to the Lebesgue measure, $\mu_{sc}$ is continuous and singular with respect to the Lebesgue measure and $\mu_d$ is a discrete measure. In view of Remark~\ref{remk:M-mu}, the function $M$ associated with $\mu$ is decomposed as
    \[M(t)=N(t)+S(t)+D(t),\quad t\in [a,b],\]
    where $N$ is absolutely continuous, $S$ is continuous and satisfies $S^\prime(t) = 0$ for a.e.\ $t \in [-1, 0]$, and $D$ is a countable sum of indicator functions multiplied by matrices. The functions $N$, $S$, and $D$ are called the \emph{absolutely continuous}, \emph{singular}, and \emph{purely discontinuous} parts of $M$, respectively.
\end{remark}

In the sequel, given $M \in \mathrm{NBV}([a, b], \mathcal M_d(\mathbb R))$, we denote by $\mu_M$ the unique Borel measure associated with $M$ in the sense of Remark~\ref{remk:M-mu}, or simply by $\mu$ when $M$ is clear from the context. In addition, when the interval $[a, b]$ is clear from the context, we denote $\Var (M)_{\rvert[a, b]}$ simply by $\Var M$.

The next result, taken from \cite{GoncalvesNetto2025Strong}, provides a sufficient condition for System~\eqref{sys:general} to be well-posed.

\begin{proposition}[{\cite[Proposition~1]{GoncalvesNetto2025Strong}}]\label{prop:wellposed}
Let $M\colon [-1, 0] \to \mathcal M_d(\mathbb R)$ be of bounded variation, define $A_M = M(0) - M(0^-)$, and suppose that $\det (\I-A_M) \ne 0$. Then, for every $\phi\in \mathcal C_0$, there exists a unique global solution $x\colon [-1,+\infty)\to \R^d$ of \eqref{sys:general} with initial condition $x_0 = \phi$.
\end{proposition}

Motivated by Proposition~\ref{prop:wellposed}, we introduce the class
\[
\mathcal W = \{M \in \mathrm{NBV}([-1, 0], \mathcal M_d(\mathbb R)) : \det (\I-A_M) \ne 0 \text{ where } A_M = M(0) - M(0^-)\}.
\]
In the sequel, with a slight abuse of notation, we will also say that a measure $\mu$ belongs to $\mathcal W$ if the function of normalized bounded variation $M$ associated with $\mu$ in the sense of Remark~\ref{remk:M-mu} belongs to $\mathcal{W}$. We shall only work from now on with System~\eqref{sys:general} under the assumption that $M \in \mathcal W$.

\subsection{Exponential behavior of solutions}

Given $M \in \mathcal W$, it is well-known (see, e.g., \cite[Section~12.3]{Hale1977Theory}) that solutions of \eqref{sys:general} satisfy an exponential bound, i.e., there exist positive constants $K,\alpha>0$ such that, for every $\phi\in \mathcal{C}_0$,
\begin{equation}
\label{eq:exp-estimate}
\norm{x_t}_\infty \leq K e^{\alpha t}\norm{\phi}_\infty,\quad  t\geq 0,
\end{equation}
where $x$ is the unique solution of \eqref{sys:general} with initial condition $\phi$. For sake of completeness, we provide a proof of \eqref{eq:exp-estimate} in Appendix~\ref{app:other}.

In order to study the stability properties of System~\eqref{sys:general}, we recall the definitions of exponential stability, exponential rate, characteristic function, and spectral abscissa.



\begin{definition}
\label{def:big-def}
Let $M \in \mathrm{NBV}([-1, 0], \mathcal M_d(\mathbb R))$.
\begin{enumerate}
\item\label{item:def:ER} For $M \in \mathcal W$, we define the \emph{exponential rate} $\omega_M$ of $M$ by
\begin{equation}
\label{eq:def-omega-M}
\begin{aligned}
\omega_M = \inf \{&\alpha \in \R: \exists K>0 \text{ s.t.\ }\norm{x_t}_\infty\leq K e^{\alpha t}\norm{\phi}_\infty, t\geq 0,\forall \phi \in \mathcal{C}_0\\
&\text{where $x$ is the unique solution of \eqref{sys:general} with } x_0=\phi\}.
\end{aligned}
\end{equation}
We say that $M$ (or $\Sigma_M$) \emph{exponentially stable} if $\omega_M < 0$.

\item We define the \emph{characteristic function} of $M$ as the function $\Delta_M \colon \mathbb C \to \mathbb C$ given, for $s \in \mathbb C$, by
\[\Delta_M(s)=\det \left(\I - \int_{-1}^0 e^{s\theta}dM(\theta)\right).\]

\item We define the \emph{spectral abscissa} $S_M$ of $M$ by
\[S_M=\sup \{\Re s : \Delta_M(s) = 0\}.\]

\item For $M \in \mathcal W$, we say that $M$ satisfies the \emph{spectrum-determined growth condition} if we have the equality
\[\omega_M = S_M.\]
\end{enumerate}
\end{definition}

A series of remarks are now in order.

\begin{remark}
\label{remk:big-remk}
\begin{enumerate}
\item Although we have defined exponential rate, characteristic function, spectral abscissa, and the spectrum-determined growth condition as properties relating to some $M \in \mathrm{NBV}([-1, 0], \mathcal M_d(\mathbb R))$, we shall often talk about these properties also for a matrix-valued measure $\mu$ on $[-1, 0]$, and in this case they should be understood as the corresponding property for the function $M$ associated with $\mu$ in the sense of Remark~\ref{remk:M-mu}.

\item As a consequence of \eqref{eq:exp-estimate}, one deduces that the set over which the infimum is computed in the definition of $\omega_M$ is nonempty, and in particular we have $\omega_M < +\infty$.

\item As the notion of exponential rate is related to solutions of \eqref{sys:general}, we require in its definition that $M \in \mathcal W$ in order to ensure the well-posedness of \eqref{sys:general} from Proposition~\ref{prop:wellposed}. However, the assumption that $M \in \mathcal W$ is not necessary in order to define the characteristic function and the spectral abscissa.

\item The inequality $S_M \leq \omega_M$ is always satisfied since, for every $s \in \mathbb C$ such that $\Delta_M(s) = 0$, denoting by $v \in \mathbb C^d$ a nonzero element in the kernel of $\I - \int_{-1}^0 e^{s\theta} dM(\theta)$, we obtain that $t \mapsto e^{s t} v$ is a ($\mathbb C^d$-valued) solution of \eqref{sys:general}, so its real and imaginary parts are $\mathbb R^d$-valued solutions of \eqref{sys:general}, and the definition of exponential rate yields that $\Re s \leq \omega_M$.

\item From \cite[Chapter~9, Theorem~4.4]{Hale1993Introduction}, it follows that a sufficient condition for a function $M \in \mathrm{NBV}([-1, 0], \mathcal M_d(\mathbb R))$ to satisfy the spectrum-determined growth condition is that the entries $M_{ij}$ of $M$ have an atom before they become constant, i.e., there is a $t_{ij} \in [-1, 0)$ with $M_{ij}(t) = M_{ij}(t_{ij}^+)$ for $t \geq t_{ij}$ and $M_{ij}(t_{ij}^-) \neq M_{ij}(t_{ij}^+)$. Note that, in this case, we necessarily have $M \in \mathcal W$.

\item\label{item:N-subset-G} Another sufficient condition for $M \in \mathrm{NBV}([-1, 0], \mathcal M_d(\mathbb R))$ to satisfy the spectrum-determined growth condition, provided in \cite[Chapter~12, Theorem~3.4]{Hale1977Theory}, is when $M$ is given by
\begin{equation}\label{eq:NS}
M(\theta) = \sum_{k=1}^{+\infty} A_k \mathbbm{1}_{[-\tau_k,0]}(\theta) + N(\theta), \quad \theta \in [-1, 0],
\end{equation}
where $\tau_k \in (0, 1)$ and $A_k\in \mathcal M_d (\R)$ for $k \in \N$ and $N$ is absolutely continuous. In this case, we also have $M \in \mathcal W$.

The discontinuity points of $M$ given by \eqref{eq:NS} are the values $-\tau_k$, $k \in \mathbb N$, which all belong to the open interval $(-1, 0)$. One can also add a discontinuity point at $-1$ in $M$ and still obtain the exact same results that we will prove in the sequel in this more general framework (namely Remark~\ref{remk:explicit-variation} and Propositions~\ref{prop:tvnosingular} and \ref{prop:positiveinf}). However, note that, since $M$ is defined only over $[-1, 0]$, such a discontinuity can only be achieved by adding a term of the form $A_0 \mathbbm 1_{(-1, 0]}$, i.e., with the indicator function of the interval $(-1, 0]$ open at $-1$.
\end{enumerate}
\end{remark}

We define the following sets of functions which will be used in the subsequent sections.
\begin{equation}
\label{eq:def-N-L-G}
\begin{aligned}
\mathcal{N} & = \{M \in \mathrm{NBV}([-1,0],\mathcal M_d (\R)) : M \text{ is given by } \eqref{eq:NS}\}, \\
\mathcal L & = \left\{M\in\mathrm{NBV}([-1,0], \mathcal M_d (\R)) : S_M < +\infty \text{ and } \inf_{\Re s \ge \alpha} \abs{\Delta_M(s)} > 0 \text{ for every } \alpha>S_M \right\}, \\
\mathcal{G} & = \{M \in \mathcal W : S_M = \omega_M \}.
\end{aligned}
\end{equation}
By Remark~\ref{remk:big-remk}.\ref{item:N-subset-G}, we have $\mathcal N \subset \mathcal G \subset \mathcal W$. In addition, we have $\mathcal N \subset \mathcal L \subset \mathcal W$. Indeed, Proposition~\ref{prop:positiveinf} in Appendix~\ref{app:other} shows the left inclusion, while, for the other one, it suffices to note that, as $s \to +\infty$ with $s \in \mathbb R$, we have $\Delta_M(s) \to \det(\I - A_M)$, with $A_M = M(0) - M(0^-)$.

With a slight abuse of notation, we will sometimes say that a measure $\mu$ belongs to one of the classes $\mathcal N$, $\mathcal L$, or $\mathcal G$ if the function of normalized bounded variation $M$ associated with $\mu$ in the sense of Remark~\ref{remk:M-mu} belongs to the same class.

\section{Continuity of the spectral abscissa in TV norm}
\label{sec:tv-norm}

In this section, we will prove that the spectral abscissa of $M$ is continuous in the total variation norm when $M$ belongs to $\mathcal L$. As a consequence, we will obtain that exponential stability of \eqref{sys:general} for some $M$ with no singular part implies exponential stability of \eqref{sys:general} for small enough perturbations of $M$ in total variation norm, still with no singular part.

We start with the following technical result relating the characteristic functions of two elements $M$ and $N$ from $\mathrm{NBV}([-1, 0], \mathcal M_d(\mathbb R))$ on a right half-plane.

\begin{lemma}\label{lemma:tvineq}
    Given $\alpha\in \R$ and $R>0$, there exists $C = C(\alpha, R) > 0$ such that
    \begin{equation}
        \sup_{\Re s \geq \alpha} \abs{\Delta_M(s)-\Delta_N(s)} \leq C \Var (M-N),
    \end{equation}
    for every $M,N\in \mathrm{NBV}([-1,0],\mathcal M_d (\R))$ such that $\Var M \leq R$ and $\Var N \leq R$.
 \end{lemma}
 
\begin{proof}
Let $\alpha_-=\max\{0,-\alpha\}$ and $r = 1 + e^{\alpha_-} R$. We set $K_r=\{A\in \mathcal M_d(\R): \norm{A}\leq r\}$. Since $\det(\cdot)$ is continuously differentiable, by the mean value theorem there exists $\widetilde{C}=\widetilde{C}(r)>0$ such that 
 \begin{equation}\label{ineq:detlipschitz}
     \abs{\det (A)-\det (B)} \leq \widetilde{C} \norm{A - B}, \quad \forall A,B\in K_r. 
 \end{equation}
Note now that $\I - \int_{-1}^0 e^{s\theta}dM(\theta)$ and $\I - \int_{-1}^0 e^{s\theta}dN(\theta)$ belong to $K_r$ for all $s \in \mathbb C$ with $\Re s \ge \alpha$, since, using Proposition~\ref{proposition:tvcoincides} from Appendix~\ref{app:measure}, we have
 \[\left\lVert\I - \int_{-1}^0 e^{s\theta}dM(\theta)\right\rVert\leq \norm{\I} + \left\lVert \int_{-1}^0 e^{s\theta}dM(\theta)\right\rVert\leq 1 + \max_{\theta \in [-1,0]}\abs{e^{s\theta}} \Var M \leq 1 + e^{\alpha_-} \Var M\leq r,\]
and similarly for $N$. Using \eqref{ineq:detlipschitz} we have that, for all $s \in \mathbb C$ with $\Re s \ge \alpha$,
\[
     \abs{\Delta_M(s)-\Delta_N(s)}\leq \widetilde{C} \left\lVert \int_{-1}^0 e^{s\theta} d(N-M)(\theta) \right\rVert \leq \widetilde{C} e^{\alpha_-} \Var (M-N),
\]
and the conclusion follows with $C = \widetilde{C} e^{\alpha_-}$.
\end{proof}
    
We are now ready to prove the main result of this section.

\begin{theorem}
\label{thm:spectral-abscissa-continuous}
The spectral abscissa map $\mathcal L \ni M \mapsto S_M\in \mathbb R \cup \{- \infty\}$ is continuous on $\mathcal L$ with respect to the total variation norm. 
\end{theorem}

\begin{proof}
    We only treat here the case $S_M \in \mathbb R$, since the case $S_M = -\infty$ can be treated exactly as in the first part of the proof below, up to replacing $S_M + \epsilon$ by $-K$ for an arbitrary $K > 0$.
    
    Let $M\in \mathcal L$ and $(N_n)_{n\in \N}\subset \mathcal L$ be a sequence such that $\lim_{n\to +\infty} \Var (M-N_n)=0$. Given $\epsilon>0$, by definition of $\mathcal L$ we have 
    \[\delta := \inf_{\Re s \geq S_M+\epsilon}\abs{\Delta_M(s)}>0.\]
    Since $\lim_{n\to +\infty} \Var (M-N_n)=0$, we may assume without loss of generality that $\Var N_n\leq 1+\Var M$ for every $n\in \N$. Let $\alpha = S_M + \epsilon$ and $R=1+\Var M$. By Lemma \ref{lemma:tvineq}, there exists $C(\alpha,R) > 0$ such that, for all $s \in \mathbb C$ with $\Re s \ge \alpha$,
    \[\delta-\abs{\Delta_{N_n}(s)}\leq \abs{\Delta_M(s)}-\abs{\Delta_{N_n}(s)}\leq \abs{\Delta_M(s)-\Delta_{N_n}(s)} \leq C(\alpha,R) \Var (M-N_n), \quad \forall n\in \N.\]
    Hence,
    \[\inf_{\Re s \ge S_M + \epsilon} \abs{\Delta_{N_n}(s)} \ge \delta - C(\alpha, R) \Var (M-N_n), \quad \forall n\in \N.\]
    Since $\Var (M-N_n)\to 0$ as $n\to +\infty$, there exists $n_1$ such that, for all $n\ge n_1$,
    \[\inf_{\Re s \ge S_M + \epsilon} \abs{\Delta_{N_n}(s)} > 0,\]
    and consequently $S_{N_n} \leq S_M+ \epsilon$ for every $n\ge n_1$.

   On the other hand, by the same argument as above, for $\beta=S_M-\epsilon$ and $R= 1+\Var M$ there exists $C(\beta,R)>0$ such that
    \[\sup_{\Re s \ge S_M-\epsilon } \abs{\Delta_M(s) - \Delta_{N_n}(s)} \leq C(\beta,R)\Var(M-N_n).\]
    That is, $\Delta_{N_n}(s)\to \Delta_M (s)$ as $n\to +\infty$, uniformly on $\Re s \ge S_M-\epsilon$. By the definition of $S_M$, there exists $s^*\in \C$ with $\Re s^* > S_M-\epsilon$ and $\Delta_M(s^*)=0$. Take $\eta>0$ small enough so that the ball described by $\abs{s-s^*}<\eta$ is contained in the half-plane described by the inequality $\Re s \ge S_M-\epsilon$. By Hurwitz's Theorem, there exists $n_2$ such that, for all $n\geq n_2$, $\Delta_{N_n}(\cdot)$ has at least one zero inside the ball $\abs{s - s^*} < \eta$, and hence $S_{N_n} \geq S_M-\epsilon$ for all $n\geq n_2$. Therefore, for every $n\geq \max \{n_1,n_2\}$ we have, 
    \[S_M-\epsilon \leq S_{N_n} \leq S_M+\epsilon,\]
    which implies that $\lim_{n\to +\infty}S_{N_n}=S_M$, completing the proof.
\end{proof}


We next deduce a consequence of the previous theorem for matrix-valued functions in $\mathcal L \cap \mathcal G$. 

\begin{corollary}\label{coro:tvstab}
If $M\in \mathcal L \cap \mathcal G$ is exponentially stable, then there exists $\epsilon>0$ such that $\widetilde M$ is exponentially stable for every $\widetilde M \in \mathcal L \cap \mathcal G$ with $\Var(M - \widetilde M) < \epsilon$.
\end{corollary}

\begin{proof}
Since $M$ is exponentially stable, we have $\omega_M < 0$.  
From the fact that $M \in \mathcal{G}$, it follows that $S_M = \omega_M < 0$. 
Hence, by Theorem~\ref{thm:spectral-abscissa-continuous}, there exists $\epsilon > 0$ such that, if $\widetilde M \in \mathcal L$ and $\Var(\widetilde{M} - M) < \epsilon$, then $S_{\widetilde{M}} < 0$. Consequently, $\widetilde{M}$ is exponentially stable for every $\widetilde{M} \in \mathcal L \cap \mathcal G$ satisfying $\Var(M - \widetilde M) < \epsilon$.
\end{proof}

\begin{remark}
\label{remk:explicit-variation}
Recalling that $\mathcal N \subset \mathcal L \cap \mathcal G$, we deduce that the previous corollary applies in particular to matrix-valued functions $M$ in $\mathcal N$, i.e., without singular part. In this case, we can explicitly compute the quantity $\Var (M - \widetilde M)$ from the statement of the corollary. Indeed, take $M, \widetilde M\in  \mathcal{N}$ and write
\[ M=\sum_{k=1}^{+\infty} A_k \mathbbm{1}_{[-\tau_k,0]}+N, \quad \widetilde M= \sum_{k=1}^{+\infty}  B_k \mathbbm{1}_{[-\widetilde \tau_k,0]}+\widetilde N,\]
where $A_k, \tau_k, N, B_k, \widetilde \tau_k, \widetilde N$ are as in \eqref{eq:NS}, assuming further, with no loss of generality, that $\tau_k \neq \tau_j$ and $\widetilde \tau_k \neq \widetilde \tau_j$ if $k \neq j$. We rewrite $M$ and $\widetilde M$ as 
\[M=\sum_{k=1}^{+\infty} \hat A_k \mathbbm{1}_{[-h_k,0]}+N, \quad \widetilde M=\sum_{k=1}^{+\infty} \hat B_k \mathbbm{1}_{[-h_k,0]}+\widetilde N,\]
where $\{h_k:k\ge 1\}=\{\tau_k:k\ge 1\} \cup \{\widetilde\tau_k: k\ge 1\}$ and
\[\hat A_k = \begin{cases}
    A_j, \text{ if } h_k = \tau_j \text{ for some $j\ge 1$,}\\
    0, \text{ otherwise},
\end{cases} \hat B_k = \begin{cases} B_j, \text{ if } h_k = \widetilde\tau_j \text{ for some $j\ge 1$,}\\
    0, \text{ otherwise}.
\end{cases}\]
So, from Proposition \ref{prop:tvnosingular} in Appendix~\ref{app:measure}, we have
\[\Var (M - \widetilde M) =  \sum_{k=1}^{+\infty} \norm{\hat A_k - \hat B_k } + \int_{-1}^0 \norm{N'(\theta)-\widetilde N' (\theta)}d\theta. \]
\end{remark}

Despite the interest of Corollary~\ref{coro:tvstab}, it cannot be applied to the case where $M$ is purely discontinuous and one perturbs the discontinuity points of $M$, as the total variation norm may increase discontinuously under such perturbations. This is illustrated by the following example.
 
\begin{example}
\label{expl:var-delay}
Let $0<\tau_1<\tau_2<1$, $A$ be a non zero element of $\mathcal M_d (\R)$ and $M,N\colon [-1,0] \to \mathcal M_d (\R)$ given by $M=A\mathbbm{1}_{[-\tau_1,0]}$ and $N=A\mathbbm{1}_{[-\tau_2,0]}$. Then
\[\Var(M-N)=2\norm{A}.\]
In this way, no matter how close $\tau_1$ is to $\tau_2$, $N$ is at a constant distance of $2 \norm{A}$ from $M$ in the sense of total variation, and in particular Corollary~\ref{coro:tvstab} does not guarantee the exponential stability of $N$ given the exponential stability of $M$.
\end{example}

The above example shows that the total variation norm is not adapted to understand if exponential stability is preserved or not with respect to perturbations in delays. That fact motivates the next section.

\section{Strong stability}
\label{sec:strong-stability}

\subsection{Strong stability of \texorpdfstring{\eqref{sys:general}}{(\ref{sys:general})} and statement of the main result}
\label{sec:main-strong-stab}

In the special case where $M$ is given by \eqref{eq:M-finite-sum} with $A = (A_1, \dotsc, A_N)\in \mathcal M_d(\R)^N$ and $\tau = (\tau_1, \dotsc, \tau_N)\in (0, 1)^N$, \eqref{sys:general} becomes the system $\Sigma(\tau)$ from \eqref{sys:difference}. As recalled in the introduction, this case was extensively studied in the literature, including the question of robustness of exponential stability with respect to perturbations in the delays. In that sense, we recall the following definition from \cite{Hale1993Introduction}.

\begin{definition}[Strong stability, \cite{Hale1993Introduction}]\label{def:strong-stab-delays}
Let $A=(A_1,\allowbreak \dotsc,\allowbreak A_N)\in \mathcal M_d (\R)^N$ and $\tau=(\tau_1,\allowbreak \dotsc,\allowbreak\tau_N) \in (0, 1)^N$.
\begin{enumerate}
\item We say that $\Sigma(\tau)$ is \emph{locally strongly stable} if there exists $\epsilon>0$ such that $\Sigma (\gamma)$ is exponentially stable for every $\gamma\in (0, 1)^N$ with $\lvert\gamma - \tau \rvert<\epsilon$.
\item We say that $\Sigma(\tau)$ is \emph{strongly stable} if $\Sigma(\gamma)$ is exponentially stable for every $\gamma\in (0, 1)^N$.
\end{enumerate}
\end{definition}

The following theorem gives a characterization of strong stability.

\begin{theorem}[Hale--Silkowski Criterion, \cite{Avellar1980Zeros, Silkowski1976Star, Hale1977Theory, Hale1993Introduction}]
\label{HS-Theorem}
Let $A=(A_1, \dotsc,A_N)\in \mathcal M_d (\R)^N$, $\widetilde \tau=(\widetilde \tau_1, \dotsc,\widetilde \tau_N)\in (0, 1)^N$ with rationally independent components, and $\tau=(\tau_1, \dotsc,\tau_N)\in (0, 1)^N$. The following statements are equivalent.

\begin{enumerate}
    \item System~$\Sigma(\widetilde{\tau})$ is exponentially stable.
    \item The Hale--Silkowski condition from \eqref{eq:HScriterion} is satisfied.
    
    
    \item System~$\Sigma(\tau)$ is locally strongly stable.
    \item System~$\Sigma(\tau)$ is strongly stable.
\end{enumerate}
\end{theorem}

\begin{remark}
The standard Hale--Silkowski criterion stated in the aforementioned references considers any delays in $\mathbb R_+^\ast$, and not only in $(0, 1)$, as we do here. As one considers only finitely many delays, it is not hard to verify that both formulations are equivalent, up to a suitable time rescaling.
\end{remark}

We now want to generalize the notion of perturbations in the delays to the more general system \eqref{sys:general}. For that purpose, define the set 
\[\mathbf{B}=\{\varphi\colon[-1,0]\to [-1,0]: \varphi \text{ is Borel-measurable}\}.\]
As stated in the introduction, our notion of perturbation of \eqref{sys:general} is System~\eqref{sys:perturbed-intro} for a perturbation $\varphi \in \mathbf B$ close enough to the identity map. Using the correspondence between matrix-valued functions $M$ of normalized bounded variation and matrix-valued Borel measures with finite total variation from Remark~\ref{remk:M-mu}, one can rewrite the perturbed system \eqref{sys:perturbed-intro} in terms of a pushforward of the measure $\mu_M$ corresponding to $M$ as
\begin{equation}
\label{sys:perturbed}
\Sigma_{\varphi_\ast\mu_M}\colon \quad x(t)=\int_{-1}^0 d(\varphi_\ast\mu_M) (\theta)x(t+\theta), \quad t\ge 0,
\end{equation}
where $\varphi_\ast\mu_M$ is the pushforward of $\mu_M$ by $\varphi\in \mathbf{B}$, i.e., the measure defined by $\varphi_\ast\mu_M(A)=\mu_M(\varphi^{-1}(A))$ for $A \in \mathcal{B}_{[-1,0]}$. Indeed, for every $x \in \mathcal C([-1, +\infty), \mathbb R^d)$ and $t\ge 0$, $x(t+\varphi(\cdot))$ is Borel-measurable and bounded, hence it is integrable with respect to $\mu_M$ and we have
\begin{equation*}
\int_{-1}^0 d(\varphi_\ast\mu_M) (\theta)x(t+\theta)=\int_{-1}^0 d\mu_M(\theta)x(t+\varphi(\theta))=\int_{-1}^0 dM(\theta)x(t+\varphi(\theta)),\quad t\ge 0.
\end{equation*}


As discussed in the introduction, in the hope of extending Theorem~\ref{HS-Theorem} for System~\eqref{sys:general}, we propose the generalization of  \eqref{eq:HScriterion} given, in terms of the quantity defined in \eqref{eq:rhoHS-intro}, by
\begin{equation}
\label{GHSC}
\rho_{\mathrm{HS}}(M) := \sup_{\xi\in \mathfrak C}  \rho \left(\int_{-1}^0 e^{i\xi(\theta)}dM(\theta)\right) < 1,
\end{equation}
where $\mathfrak C=\{\xi\colon [-1,0]\to \R : \xi\text{ is Borel-measurable}\}$.

\begin{remark}
By Proposition~\ref{proposition:tvcoincides} in Appendix~\ref{app:measure}, we immediately have from \eqref{GHSC} that $\rho_{\mathrm{HS}}(M) \leq \Var(M)$. In addition, one has equality in dimension $d = 1$ since, in that case, denoting by $P$ and $N$ the positive and negative sets of $\mu_M$ in the sense of the Hahn decomposition with $P \cup N = [-1, 0]$ and $P \cap N = \emptyset$ (see, e.g., \cite[Theorem~3.3]{Folland1999}), we have the equality $\Var(M) = \int_{-1}^0 e^{i\xi(\theta)}dM(\theta)$ for the function $\xi\colon [-1,0]\to \R$ given by $\xi(\theta) = 0$ if $\xi \in P$ and $\xi(\theta) = \pi$ if $\xi \in N$.
\end{remark}

Before continuing, we must also adapt Definition \ref{def:strong-stab-delays} of strong stability to the case of System \eqref{sys:general}.

\begin{definition}[Strong stability]
\label{def:strong-stab}
Let $M \in \mathcal G$ and $\mu$ be the matrix-valued Borel measure associated with $M$ in the sense of Remark~\ref{remk:M-mu}.
\begin{enumerate}[ref={\arabic*)}]
\item We say that $\Sigma_M$ is \emph{locally strongly stable} if there exists $\epsilon>0$ such that $\varphi_\ast \mu$ is exponentially stable for all $\varphi \in \mathbf B$ with $\lVert \varphi - \id \rVert_\infty < \epsilon$  and $\varphi_\ast \mu \in \mathcal G$.
\item We say that $\Sigma_M$ is \emph{strongly stable} if $\varphi_\ast \mu$ is exponentially stable for all $\varphi\in\mathbf{B}$ such that $\varphi_\ast \mu \in \mathcal G$.
\end{enumerate}
\end{definition}

We then conjecture the following general result.
\begin{conjecture}
\label{GHS-Conjecture}
Let $M \in \mathcal G$. Then the following statements are equivalent.
    \begin{enumerate}
        \item We have $\rho_{\mathrm{HS}}(M) <1$.
        \item System~$\Sigma_M$ is strongly stable.
        \item System~$\Sigma_M$ is locally strongly stable.
    \end{enumerate}
\end{conjecture}

\begin{remark}
Conjecture~\ref{GHS-Conjecture} was proved in the one-dimensional case in \cite{GoncalvesNetto2025Strong} in the slightly modified setting in which $\mathcal G$ is replaced by $\mathcal W$ in the statement of the conjecture and in the definitions of strong stability. The same argument from the proof of \cite[Theorem~2]{GoncalvesNetto2025Strong} also shows that, if $M \in \mathcal W$ and $\Var(M) < 1$, then $M$ is strongly stable (and hence also locally strongly stable), changing once again $\mathcal G$ for $\mathcal W$ in the definition of strong stability.
\end{remark}

We shall prove in the sequel a partial result on Conjecture~\ref{GHS-Conjecture}, in which we replace the strong stability of $\Sigma_M$ by slightly stronger notions which take into account a uniform exponential rate. Recall that, according to Definition~\ref{def:big-def}, given \(M \in \mathrm{NBV}([-1, 0], \mathcal M_d(\mathbb R))\), the exponential rate \(\omega_M\) associated with System~\(\Sigma_M\) from \eqref{sys:general} is the quantity defined in \eqref{eq:def-omega-M}, which represents the least upper bound on the exponential behavior of solutions of \eqref{sys:general}.

\begin{definition}[Strong stability with uniform rate]
\label{def:stron-stab-UR}
Let $M \in \mathcal G$ and $\mu$ be the matrix-valued Borel measure associated with $M$ in the sense of Remark~\ref{remk:M-mu}.
\begin{enumerate}    
\item We say that
$\Sigma_M$ is \emph{locally strongly stable with uniform rate} if there exist $\alpha > 0$ and $\epsilon > 0$ such that $\omega_{\varphi_\ast\mu}\leq -\alpha$ for all $\varphi \in \mathbf B$ with $\lVert \varphi - \id \rVert_\infty < \epsilon$  and $\varphi_\ast \mu \in \mathcal G$. 
\item We say that $\Sigma_M$ is \emph{strongly stable with uniform rate} if there exists $\alpha>0$ such that $\omega_{\varphi_\ast\mu}\leq -\alpha$ for all $\varphi\in\mathbf{B}$ with $\varphi_\ast \mu \in \mathcal G$.
\end{enumerate}
\end{definition}

The above definition is well justified in view of the following theorem, which shows that \eqref{GHSC} is equivalent to strong stability with uniform rate and local strong stability with uniform rate.

\begin{theorem}\label{GHS-Theorem}
Let $M \in \mathcal G$. Then the following statements are equivalent.
    \begin{enumerate}
        \item We have $\rho_{\mathrm{HS}}(M) < 1$.
        \item System~$\Sigma_M$ is strongly stable with uniform rate.
        \item System~$\Sigma_M$ is locally strongly stable with uniform rate.
    \end{enumerate}
\end{theorem}

Before we proceed with the proof of Theorem~\ref{GHS-Theorem}, let us provide a brief illustrative example of its application.

\begin{example}
Let us consider \eqref{sys:general} with \(M\) given by \eqref{eq:NS} under the assumption that \(N\) piecewise affine. More precisely, we fix a sequence \((A_k)_{k \in \mathbb N}\) in \(\mathcal M_d(\mathbb R)\), a sequence \((\tau_k)_{k \in \mathbb N}\) of pairwise distinct real numbers in \((0, 1)\), a positive integer \(N \in \mathbb N\), \(N\) matrices \(B_1, \dotsc, B_N\) in \(\mathcal M_d(\mathbb R)\), and real numbers \(h_0, \dotsc, h_N\) such that \(0 = h_0 < h_1 < \dotsb < h_N = 1\). We set \(\Delta_j = h_j - h_{j - 1}\) for \(j \in \{1, \dotsc, N\}\) and we consider the system
\begin{equation}
\label{eq:expl-two-matrices}
x(t) = \sum_{k = 1}^{+\infty} A_k x(t - \tau_k) + \sum_{j = 1}^N B_j \int_{-{h_j}}^{-h_{j-1}} x(t + \theta) d\theta,
\end{equation}
which is of the form \eqref{sys:general} with
\[M(\theta) = \sum_{k = 1}^{+\infty} A_k \mathbbm 1_{[-\tau_k, 0]}(\theta) + \sum_{j = 1}^N B_j  \min(\Delta_j, \max(\theta + h_j, 0)).\]
From \eqref{GHSC}, we have
\[
\rho_{\mathrm{HS}}(M) = \sup_{\xi \in \mathfrak C} \rho\left(\sum_{k = 1}^{+\infty} A_k e^{i \xi(-\tau_k)} + \sum_{j = 1}^N B_j \int_{-h_j}^{-h_{j-1}} e^{i \xi(\theta)} d\theta\right).
\]
Since the integrals in the above expression do not change when \(\xi\) is modified in a set of Lebesgue measure zero, one may regard the values \(\xi(-\tau_k)\) as independent variables, that is, we have
\[
\rho_{\mathrm{HS}}(M) = \sup_{((\theta_k)_{k \in \mathbb N}, \xi) \in [0, 2\pi]^{\mathbb N} \times \mathfrak C} \rho\left(\sum_{k = 1}^{+\infty} A_k e^{i \theta_k} + \sum_{j = 1}^N B_j \int_{-h_j}^{-h_{j-1}} e^{i \xi(\theta)} d\theta\right).
\]
Any complex number of the form \(\int_{-h_j}^{-h_{j-1}} e^{i \xi(\theta)} d\theta\) for \(\xi \in \mathfrak C\) belongs necessarily to the closed disk \(\overline B(0, \Delta_j)\) centered at \(0\) and with radius \(\Delta_j\) and, conversely, any complex number in this disk can be written under the previous integral form through a suitable choice of Borel-measurable function \(\xi\colon (-h_j, -h_{j-1}) \to \mathbb R\). Hence
\[
\rho_{\mathrm{HS}}(M) = \sup_{((\theta_k)_{k \in \mathbb N}, \zeta_1, \dotsc, \zeta_N) \in [0, 2\pi]^{\mathbb N} \times \overline B(0, \Delta_1) \times \dotsb \times \overline B(0, \Delta_N)} \rho\left(\sum_{k = 1}^{+\infty} A_k e^{i \theta_k} + \sum_{j = 1}^N B_j \zeta_j\right).
\]
Using Lemma~\ref{lemma:max_principle} in Appendix~\ref{app:other}, we deduce that
\[
\rho_{\mathrm{HS}}(M) = \sup_{((\theta_k)_{k \in \mathbb N}, (\vartheta_j)_{j = 1}^N) \in [0, 2\pi]^{\mathbb N} \times [0, 2 \pi]^N} \rho\left(\sum_{k = 1}^{+\infty} A_k e^{i \theta_k} + \sum_{j = 1}^N B_j \Delta_j e^{i\vartheta_j}\right).
\]
Hence, by Theorem~\ref{GHS-Theorem}, System~\eqref{eq:expl-two-matrices} is strongly stable with uniform rate if and only if
\[\sup_{((\theta_k)_{k \in \mathbb N}, (\vartheta_j)_{j = 1}^N) \in [0, 2\pi]^{\mathbb N} \times [0, 2 \pi]^N} \rho\left(\sum_{k = 1}^{+\infty} A_k e^{i \theta_k} + \sum_{j = 1}^N B_j \Delta_j e^{i\vartheta_j}\right) < 1.\]
\end{example}


The remainder of this section is dedicated to the proof of Theorem~\ref{GHS-Theorem}.

\subsection{Properties of \texorpdfstring{$\rho_{\mathrm{HS}}$}{rhoHS}}


We start by proving some properties of the quantity $\rho_{\mathrm{HS}}$ defined in \eqref{eq:rhoHS-intro}. Our first result concerns the particular case in which $M$ is given by \eqref{eq:M-finite-sum} and $\rho_{\mathrm{HS}}$ reduces to the left-hand side of \eqref{eq:HScriterion}. In this case, as $\theta_1, \dotsc, \theta_N$ are only used in $e^{i \theta_1}, \dotsc, e^{i \theta_N}$ in \eqref{eq:HScriterion}, the maximum in this expression can be seen as a maximum over $N$ complex numbers in the unit circle. Using a suitable version of the maximum principle for the spectral radius matrix-valued functions, one can replace the maximum in \eqref{eq:HScriterion} for a maximum over $N$ complex numbers in the unit disk. More precisely, we have the following result.

\begin{proposition}
\label{prop:max_principle}
    For $A=(A_1, \dotsc,A_N)\in \mathcal M_d (\R)^N$, we have
    \[\max_{t\in \overline B(0, 1)^N} \rho \left( \sum_{k=1}^N t_k A_k  \right) = \max_{\theta\in[0,2\pi]^N} \rho \left( \sum_{k=1}^N A_k e^{i\theta_k} \right),\]
    where $\overline B(0, 1)=\{s\in \C : \abs{s}\leq 1\}$.
\end{proposition}

\begin{proof}
  We take $f(z_1, \dotsc,z_N)=\sum_{k=1}^N z_k A_k $ and apply Lemma~\ref{lemma:max_principle} from Appendix~\ref{app:other}.
\end{proof}

Our next aim is to generalize Proposition~\ref{prop:max_principle} by proving that the function $e^{i \xi(\cdot)}$ in \eqref{eq:rhoHS-intro} can be replaced by a function taking values in the unit disk. For that purpose, we define the set
\begin{equation}
\label{eq:def-mathfrak-D}
\mathfrak D = \{\eta\colon [-1,0]\to \bar{B}_1(0) : \eta \text{ is Borel-measurable}\}.
\end{equation}
Our first result in that direction is the following.

    \begin{lemma}
    \label{lemm:partial_max_principle}
    Let $M\in \mathrm{NBV}([-1, 0],\mathcal M_d (\R))$ and suppose that 
   \[\sup_{\xi\in \mathfrak C}  \rho \left(\int_{-1}^0 e^{i\xi(\theta)} dM(\theta)\right) <1.\]
    Then 
    \[\sup_{\eta\in \mathfrak D} \rho \left(\int_{-1}^0 \eta(\theta)dM(\theta)\right) \leq 1.\]
\end{lemma}
\begin{proof}
    By contradiction, suppose that 
    \[\sup_{\eta\in \mathfrak D} \rho \left(\int_{-1}^0 \eta(\theta)dM(\theta)\right) > 1.\]
    Then there exists $\eta \in \mathfrak D$ such that $\rho \left(\int_{-1}^0 \eta(\theta)dM(\theta)\right) > 1$. From \cite[Theorem~2.10]{Folland1999} there exists a sequence $(\eta_n)_{n \in \mathbb N}$ of simple functions such that $\abs{\eta_n(\theta)} \leq \abs{\eta(\theta)}$ for every $\theta\in [-1,0]$ and $n \in \mathbb N$ and $\eta_n \to \eta$ pointwise as $n \to +\infty$. Then, from Lebesgue's dominated convergence theorem, we obtain
    \[\lim_{n \to +\infty} \int_{-1}^0 \eta_n d\mu_M = \int_{-1}^0 \eta d\mu_M. \]
    By the continuity of the spectral radius, we obtain that there exists a simple function $\widetilde\eta \in \mathfrak D$ satisfying $\rho \left(\int_{-1}^0 \widetilde\eta(\theta)dM(\theta)\right) > 1$. Let us write $\widetilde\eta = \sum_{k=1}^m\alpha_k \mathbbm{1}_{E_k}$ with $m \in \mathbb N$, $\alpha_k \in \bar{B}_1(0)$ and $E_k \in \mathcal B_{[-1,0]}$ for $k \in \{1, \dotsc, m\}$ such that $\bigcup_{k=1}^m E_k=[-1,0]$ and $E_j\cap E_k=\emptyset$ for $j\ne k$. Using Proposition~\ref{prop:max_principle}, we have
    \begin{align*}
        1<\rho \left(\int_{-1}^0 \widetilde\eta(\theta)dM(\theta)\right) &=\rho \left(\sum_{k=1}^m \alpha_k \mu_M(E_k) \right)\leq \max_{t\in \bar{B}_1(0)^m} \rho \left( \sum_{k=1}^m t_k \mu_M(E_k)  \right)\\
        &=  \max_{\theta\in[0,2\pi]^m} \rho \left( \sum_{k=1}^m \mu_M(E_k) e^{i\theta_k} \right)\leq \sup_{\xi\in \mathfrak C}  \rho \left(\int_{-1}^0 e^{i\xi(\theta)} dM(\theta)\right)<1,
    \end{align*}
    which gives the desired contradiction. 
    \end{proof} 

We are now in position to generalize Proposition~\ref{prop:max_principle} to $\rho_{\mathrm{HS}}$.

\begin{proposition}\label{prop:max_principle1}
For $M \in \mathrm{NBV}([-1, 0],\mathcal M_d (\R))$, we have
\begin{equation}
\label{eq:supC-equal-supD}
\rho_{\mathrm{HS}}(M) = \sup_{\xi\in \mathfrak C}  \rho \left(\int_{-1}^0 e^{i\xi(\theta)} dM(\theta)\right)=\sup_{\eta\in \mathfrak D} \rho \left(\int_{-1}^0 \eta(\theta)dM(\theta)\right).
\end{equation}
\end{proposition}
    
    \begin{proof}
        Set $\rho_0= \rho_{\mathrm{HS}}(M)$. Note that, for every $\epsilon > 0$,
    \[\sup_{\xi\in \mathfrak C}  \rho \left(\int_{-1}^0 e^{i\xi(\theta)} d\left( \frac{M(\theta)}{\rho_0+\epsilon}\right)\right) = \frac{\rho_0}{\rho_0+\epsilon} < 1.\]
    Therefore, by Lemma~\ref{lemm:partial_max_principle}, we have
    \[\sup_{\eta\in \mathfrak D}  \rho \left(\int_{-1}^0 \eta(\theta) d\left( \frac{M(\theta)}{\rho_0+\epsilon}\right)\right) \leq 1,\]
    hence
    \[\sup_{\eta\in \mathfrak D} \rho \left(\int_{-1}^0 \eta(\theta)dM(\theta)\right) \leq \rho_0 + \epsilon,\quad \forall \epsilon > 0.\]
    Therefore, as $\epsilon > 0$ is arbitrary, we deduce that
    \[\sup_{\xi\in \mathfrak C}  \rho \left(\int_{-1}^0 e^{i\xi(\theta)} dM(\theta)\right)\leq \sup_{\eta\in \mathfrak D} \rho \left(\int_{-1}^0 \eta(\theta)dM(\theta)\right) \leq \rho_0 = \sup_{\xi\in \mathfrak C}  \rho \left(\int_{-1}^0 e^{i\xi(\theta)} dM(\theta)\right), \]
    as required.
    \end{proof}

Our next result shows that the supremum over $\mathfrak D$ in the right-hand side of \eqref{eq:supC-equal-supD} is attained. Although this result is not directly used in the proof of Theorem~\ref{GHS-Theorem}, it can be obtained as a byproduct of our analysis and we provide it here since we deem that it has an interest on itself.

    \begin{proposition}
        If $M\in \mathrm{NBV}([-1, 0],\mathcal M_d (\R))$, then there exists $\eta_\ast\in \mathfrak D$ such that
        \[\sup_{\eta\in \mathfrak D} \rho \left(\int_{-1}^0 \eta(\theta)dM(\theta)\right)=\rho \left(\int_{-1}^0 \eta_\ast(\theta)dM(\theta)\right).\]
    \end{proposition}
    \begin{proof}
        Consider a sequence $(\eta_n)_{n \in \mathbb N}$ in $\mathfrak D$ such that 
        \[
        \lim_{n\to +\infty}\rho \left(\int_{-1}^0 \eta_n(\theta)dM(\theta)\right)=\sup_{\eta\in \mathfrak D} \rho \left(\int_{-1}^0 \eta(\theta)dM(\theta)\right).
        \]
        Let $\mu$ be the matrix-valued measure associated with $M$. We know that each coordinate $\mu_{ij}$ of $\mu$ can be written as $\mu_{ij}=\mu_{ij}^+-\mu_{ij}^-$, where $\mu_{ij}^+$ and $\mu_{ij}^-$ are positive measures. For $K=2 d^2$, we consider $\mu_1, \dotsc,\mu_K$ an enumeration of these $\mu_{ij}^+,\mu_{ij}^-$. Denoting the measurable space $\mathcal X=([-1,0],\mathcal{B}_{[-1,0]})$, we have $\eta_n\in L^\infty(\mu_k):=L^\infty(\mathcal X,\C,\mu_k)$ and $\norm{\eta_n}_{L^\infty(\mu_k)}\leq 1$ for every $k\in\{1, \dotsc,K\}$. By using the fact that the dual space $(L^1(\mu_k))^*$ is equal to $L^\infty(\mu_k)$, considering Banach--Alaoglu--Bourbaki theorem and after taking subsequences, we deduce the existence of $n_1, \dotsc,n_K\in \mathfrak D$ such that
        \[\eta_n\xrightharpoonup[n \to +\infty]{*} n_k \text{ in } L^\infty(\mu_k),\quad \forall k\in\{1, \dotsc,K\}.\]
        By Proposition~\ref{prop:measure2} in Appendix~\ref{app:measure}, there exists $\eta_\ast\in \mathfrak D$ such that for all $ k\in\{1, \dotsc,K\}$, we have $\eta_n\xrightharpoonup{*} \eta_\ast$ in $L^\infty(\mu_k)$ as $n \to +\infty$, that is,
        \[\lim_{n\to +\infty}\int_{-1}^0\eta_n\phi d\mu_k =\int_{-1}^0\eta_\ast\phi d\mu_k,\quad \forall \phi \in L^1(\mu_k).\]
        Taking $\phi\equiv1$ and using the continuity of $\rho (\cdot)$, we obtain
        \[ \rho\left(\int_{-1}^0 \eta_\ast(\theta)d\mu(\theta)\right)=\lim_{n \to +\infty}\rho \left(\int_{-1}^0 \eta_n(\theta)d\mu(\theta)\right)=\sup_{\eta\in \mathfrak D} \rho \left(\int_{-1}^0 \eta(\theta)dM(\theta)\right),\]
        as required.
        \end{proof}
        

\subsection{Characteristic function of perturbed systems}
\label{sec:Delta}

Out next preparatory results for the proof of Theorem~\ref{HS-Theorem} concern properties of the characteristic functions of perturbed systems of the form \eqref{sys:perturbed-intro}. We first prove that the condition $\rho_{\mathrm{HS}}(M) < 1$ ensures that such characteristic functions admit no roots in a neighborhood of the imaginary axis, whose size is independent of the perturbation $\varphi$.

\begin{lemma}\label{lemma:noroots}
Let $M\in \mathrm{NBV}([-1, 0],\mathcal M_d (\R))$, denote by $\mu$ the matrix-valued Borel measure associated with $M$, and suppose that $\rho_{\mathrm{HS}}(M)<1$. Then there exists $\nu>0$ such that, for all $\varphi \in \mathbf B$, there are no roots of the characteristic function of $\varphi_\ast\mu$
on the vertical strip $\{a+ bi\in \C : -\nu \leq a\leq \nu,\, b\in \R\}$.
\end{lemma}

\begin{proof}
Note that, for a given $\varphi \in \mathbf B$, using the properties of a pushforward measure, we have that the characteristic function $\Delta_\varphi$ of $\varphi_\ast \mu$ is given by
\[\Delta_\varphi(s) = \det\left(\I-\int_{-1}^0 e^{s\varphi(\theta)}dM(\theta) \right), \qquad s \in \mathbb C.\]

Consider the set 
\[\mathcal{T} = \left\{T_{\varphi,b}(a)\in \mathcal{M}_d(\R): -1 \leq a \leq 1,\, b\in \R,\text{ and } \varphi \in \mathbf B\right\},\]
where
\[
T_{\varphi,b}(a)=\int_{-1}^0 e^{(a + i b)\varphi(\theta)}dM(\theta).
\]
        Note that $\mathcal{T}$ is a bounded set since, for any $T_{\varphi,b}(a)\in \mathcal{T}$,
        \[\norm{T_{\varphi,b}(a)}\leq \left(\Var(M)_{\rvert [-1,0]}\right)\sup_{\theta\in[-1,0]}\abs{e^{(a+ib)\varphi(\theta)}} \leq \left(\Var(M)_{\rvert [-1,0]}\right) e.\]
        Therefore, the closure $\cl\mathcal{T}$ of $\mathcal{T}$ is compact and the restriction $\rho_{\rvert \cl\mathcal{T}}$ of the spectral radius function to $\cl\mathcal T$ is uniformly continuous. Let
        \[\epsilon=\frac{1-\rho_{\mathrm{HS}}(M)}{2}.\]
        Then there exists $\delta>0$ such that, for any $A,B\in \cl \mathcal{T}$ such that $\norm{A-B}<\delta$, we have
        \[\rho(A)<\rho(B)+\epsilon.\]
        On the other hand, for $T_{\varphi,b}(a),T_{\varphi,b}(0)\in \mathcal{T}$,
     \begin{align*}
            \lVert T_{\varphi,b}(a)-T_{\varphi,b}(0)\rVert&=\left\lVert\int_{-1}^0(e^{a\varphi(\theta)}-1)e^{ib\varphi(\theta)}dM(\theta)\right\rVert\leq \Var(M)_{\rvert [-1,0]}\sup_{\theta\in [-1,0]}\abs{e^{a \varphi(\theta)} - 1}\\
            &\leq \Var(M)_{\rvert [-1,0]} \abs{e^{-a} - 1}.
        \end{align*}  
        Let $\nu \in (0, 1]$ be such that, for every $a\in [-\nu,\nu]$, it holds $\Var(M)_{\rvert [-1,0]} \abs{e^{-a} - 1} < \delta$. Thus 
        \[\rho(T_{\varphi,b}(a))<\rho(T_{\varphi,b}(0))+\epsilon\leq
        \frac{\rho_{\mathrm{HS}}(M)+1}{2}<1, \]
        for every $a\in [-\nu,\nu]$, $b\in \R$, and $\varphi \in \mathbf B$. Therefore, $\I - T_{\varphi,b}(a)$ is invertible for all $a\in [-\nu,\nu]$, $b\in \R$, and $\varphi \in \mathbf B$, resulting in
        \[\det \left(\I - T_{\varphi,b}(a)\right) = \det\left(\I- \int_{-1}^0 e^{(a + i b) \varphi(\theta)} dM(\theta)\right)\ne 0. \qedhere\]
        \end{proof}

The last technical result on characteristic functions of perturbed systems of the form \eqref{sys:perturbed-intro} that we shall need is the following.

\begin{lemma}\label{lemma:roots}
Let $M \in \mathrm{NBV}([-1, 0],\mathcal M_d (\R))$, denote by $\mu$ the matrix-valued Borel measure associated with $M$, and suppose that $\rho_{\mathrm{HS}}(M) \geq 1$.
Then, for every $\epsilon > 0$ and $\delta > 0$, there exists $\varphi \in \mathbf B$ such that $\norm{\varphi-\id}_\infty < \epsilon$, $\varphi_\ast\mu\in \mathcal G$, and such that the characteristic function of $\varphi_\ast \mu$
admits a root $s$ with $\Re s \geq \ln \rho_{\mathrm{HS}}(M) - \delta$.
\end{lemma}

\begin{proof}
To simplify the notations in the proof, set $\rho_0 = \rho_{\mathrm{HS}}(M)$. Let $\epsilon > 0$ and $\delta > 0$ be given. 
Note that it suffices to prove the result for $\epsilon>0$ and $\delta>0$ small enough, and thus we assume, without loss of generality, that $\epsilon < 1$, and $\delta < 2 \ln \rho_0$ in the case $\rho_0 > 1$. We set
\begin{equation}
\label{eq:delta-1}
\delta_1 = \begin{dcases*}
\frac{\delta}{2} & if $\rho_0 > 1$, \\
\frac{\epsilon \delta}{16} & if $\rho_0 = 1$.
\end{dcases*}
\end{equation}

By Proposition~\ref{prop:max_principle1}, there is a simple function $\eta_0\in \mathfrak D$ such that
\begin{equation}\label{eq:lemma:expinstab1}
\rho\left(\int_{-1}^0 \eta_0 (\theta)dM(\theta)\right)\ge \rho_0 e^{-\delta_1}.
\end{equation}
Without loss of generality, we write $\eta_0$ as
\[\eta_0=\sum_{k=1}^N \xi_k\mathbbm{1}_{E_k},\]
with $N \in \mathbb N$, $E_k \in \mathcal B_{[-1, 0]}$ and $\sup E_k - \inf  E_k<\frac{\epsilon}{2}$ for all $k \in \{1, \dotsc, N\}$, and $\max_k \abs{\xi_k}\leq 1$. Now we take $\tau_1, \dotsc,\tau_N\in (0, 1)$ rationally independent so that, for all $k \in \{1, \dotsc, N\}$, we have
\[ -\tau_k \in \left(\max\left\{\inf E_k-\frac{\epsilon}{4}, -1\right\}, \min\left\{\sup E_k+\frac{\epsilon}{4}, -\frac{\epsilon}{8}\right\}\right).\]
Note that all the above intervals have length at least $\frac{\epsilon}{8}$ and that
\begin{equation}
\label{eq:bound-tau}
\tau_k > \frac{\epsilon}{8} \quad \text{ for all } k \in \{1, \dotsc, N\}.
\end{equation}

We define $\varphi \colon [-1,0]\to [-1,0]$ so that $\varphi_{\rvert E_k}=-\tau_k$ for all $k \in \{1, \dotsc, N\}$. It clearly holds that $\varphi \in \mathbf B$, $\norm{\varphi - \id}_\infty<\epsilon$, $\varphi_\ast \mu \in \mathcal N \subset \mathcal G$, and, if we set $A_k=\int_{E_k} dM(\theta)$ for every $k\in \{1, \dotsc,N\}$, then the perturbed system \eqref{sys:perturbed-intro}
can be simply written as 
\begin{equation}
\label{eq:lemma:expinstab-sum}
x(t)=\sum_{k=1}^N A_k x(t-\tau_k).
\end{equation}
On the other hand, from \eqref{eq:lemma:expinstab1} we obtain
\[\rho\left(\sum_{k=1}^N \xi_k A_k\right)\ge \rho_0 e^{-\delta_1}.\]
Set $\beta = \rho_0 e^{-\delta_1}$ and $\alpha = \ln \rho_0 - \frac{\delta}{2}$. It follows from \eqref{eq:delta-1} and \eqref{eq:bound-tau} that $0 < \frac{e^{\alpha \tau_k}}{\beta} \leq 1$ for every $k \in \{1, \dotsc, N\}$ and 
\[\rho\left(\sum_{k=1}^N \frac{\xi_ke^{\alpha \tau_k}}{\beta} e^{-\alpha\tau_k}A_k\right)\ge \frac{\rho_0 e^{-\delta_1}}{\beta} = 1.\]
From Proposition~\ref{prop:max_principle}, we have
\[\max_{\theta\in[0,2\pi]^N} \rho \left( \sum_{k=1}^N e^{-\alpha\tau_k}A_k e^{i\theta_k} \right) \geq 1.\]
Thus, from Theorem~\ref{HS-Theorem}, the system 
\begin{equation}\label{eq:lemma:expinstab2}
y(t)=\sum_{k=1}^N e^{-\alpha \tau_k}A_k y(t-\tau_k)
\end{equation}
is not exponentially stable, implying that its characteristic function admits a root $z$ with $\Re z > -\frac{\delta}{2}$. Note that Systems~\eqref{eq:lemma:expinstab-sum} and \eqref{eq:lemma:expinstab2} are related by the change of variables $x(t) = e^{\alpha t} y(t)$, and hence the spectrum of \eqref{eq:lemma:expinstab-sum} is simply a translation of that of \eqref{eq:lemma:expinstab2} by $\alpha$. In particular, the characteristic function of  $\varphi_\ast \mu$ admits a root $s$ with $\Re s > \alpha - \frac{\delta}{2} = \ln \rho_0 - \delta$.
\end{proof}

\subsection{Proof of Theorem~\ref{GHS-Theorem} and further comments}

    \begin{proof}[Proof of Theorem \ref{GHS-Theorem}]
        It is clear from the definition that 2.\ implies 3.
        
        To show that 1.\ implies 2., let $\varphi \in \mathbf B$ be such that $\varphi_\ast \mu_M \in \mathcal{G}$. From Lemma~\ref{lemma:noroots}, there is $\alpha>0$ such that there are no roots of the characteristic function $\Delta_\varphi$ of $\varphi_* \mu_M$, namely
        \[\Delta_\varphi(s)=\det\left(\I-\int_{-1}^0e^{s\varphi(\theta)}dM(\theta)\right) \qquad s \in \mathbb C,\]
        on the vertical strip $\{a+ib\in \C : a\in [-\alpha, \alpha],\, b\in\R\}$. 
        We claim that $s=a+ib$ is not a root of $\Delta_\varphi$ if $a\ge 0$. Indeed, given $s = a + i b$ with $a \geq 0$, for $\eta\colon[-1,0]\to \bar{B}_1(0)$ given by $\eta(\theta)=e^{s\varphi(\theta)}$, it is clear that $\eta \in \mathfrak D$, so from the hypothesis and Proposition~\ref{prop:max_principle1} we obtain
        \[ \rho \left(\int_{-1}^0 e^{s\varphi(\theta)}dM(\theta)\right)\leq\sup_{\eta\in \mathfrak D} \rho \left(\int_{-1}^0 \eta(\theta)dM(\theta)\right)<1.\]
        This implies that $1$ is not an eigenvalue of $\int_{-1}^0 e^{s\varphi(\theta)}dM(\theta)$, i.e., $s$ is not a root of $\Delta_\varphi$. Hence,  $\omega_{\varphi_\ast\mu}\leq -\alpha$.

        To show that 3.\ implies 1., we argue by contradiction and hence assume that 1.\ does not hold true, i.e., $\rho_{\mathrm{HS}}(M) \geq 1$. Since $M$ is locally strongly stable with uniform rate, we can find $\alpha>0$ and $\epsilon>0$ such that, for all $\varphi\in \mathbf{B}$ with $\lVert\varphi - \id\rVert_{\infty}<\epsilon$ and $\varphi_\ast\mu_M\in \mathcal{G}$, all the roots $s$ of 
        \[\Delta_\varphi(s)=\det\left(\I-\int_{-1}^0e^{s\varphi(\theta)}dM(\theta)\right)\]
        satisfy $\Re s \leq -\alpha$. This is in contradiction with Lemma~\ref{lemma:roots}.
        \end{proof}

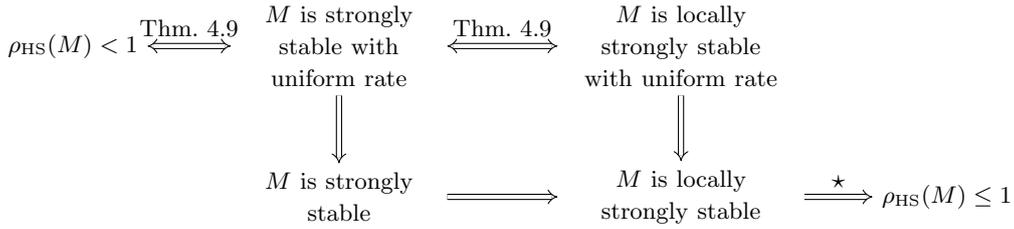
\begin{figure}[ht]
\centering
\begin{tikzpicture}
\node (rhoHS) at (0, 0) {\small $\rho_{\mathrm{HS}}(M) < 1$};
\node[text width=2.6cm, align = center] (SSUR) at (3.5, 0) {\small $M$ is strongly stable with uniform rate};
\node[text width=3cm, align = center] (LSSUR) at (8, 0) {\small $M$ is locally strongly stable with uniform rate};
\node[text width=2.6cm, align = center] (SS) at (3.5, -2) {\small $M$ is strongly stable};
\node[text width=3cm, align = center] (LSS) at (8, -2) {\small $M$ is locally strongly stable};
\node (rhoHSleq) at (11.5, -2) {\small $\rho_{\mathrm{HS}}(M) \leq 1$};

\draw[double equal sign distance, double, Implies-Implies] (rhoHS) --node[midway, above] {\small Thm.~\ref{GHS-Theorem}} (SSUR);
\draw[double equal sign distance, double, Implies-Implies] (SSUR) --node[midway, above] {\small Thm.~\ref{GHS-Theorem}} (LSSUR);
\draw[double equal sign distance, double, -Implies] (SSUR) -- (SS);
\draw[double equal sign distance, double, -Implies] (LSSUR) -- (LSS);
\draw[double equal sign distance, double, -Implies] (SS) -- (LSS);
\draw[double equal sign distance, double, -Implies] (LSS) --node[midway, above] {$\star$} (rhoHSleq);
\end{tikzpicture}
\caption{Diagram of the results of the paper on Conjecture~\ref{GHS-Conjecture}.}
\label{fig:results}
\end{figure}

Regarding Conjecture~\ref{GHS-Conjecture}, the results presented in this paper allow one to obtain the diagram given in Figure~\ref{fig:results}. Indeed, by Theorem~\ref{GHS-Theorem}, we have that \eqref{GHSC} implies strong stability with uniform rate, and then strong stability, and finally local strong stability. 
An argument for the implication $\star$ goes by contraposition as follows. If $\rho_{\mathrm{HS}}(M) > 1$, then by taking $\delta=\frac{\ln\rho_{\mathrm{HS}}(M)}{2}$ we can find, according to Lemma~\ref{lemma:roots},
a function $\varphi$ arbitrarily close to the identity function
such the corresponding characteristic function has a root with positive real part, hence $M$ cannot be locally strongly stable.



\appendix

\section{Appendix}
\label{sec:appendix}

\subsection{Results in measure theory}
\label{app:measure}

We gather here some measure-theoretical results needed in the paper. The first one asserts that some classical properties for scalar-valued functions of bounded variation and their associated Borel measures remain true in the matrix-valued context of this paper. In order to state that result, we set 
\[
\norm{\mu}_c([a,b])=\sup\left\{\sum_{j=1}^{+\infty} \norm{\mu(E_j)} :  E_1, E_2, \dotsc \in \mathcal{B}_{[a,b]} \text{ are disjoint and } \bigcup_{j=1}^{+\infty} E_j = [a,b]\right\}.
\]
We then have the following proposition (see also \cite[Chapter~3, Theorem~5.6]{Gripenberg1990} for the second part of the statement).

\begin{proposition}\label{proposition:tvcoincides}
    Let $M\colon [a, b]\to \mathcal M_d (\R)$ be of normalized bounded variation and $\mu_M$ be the measure associated with $M$ in the sense of Remark~\ref{remk:M-mu}. Then
    \begin{align}
    \label{eq:variations-coincide}
    \Var (M)_{\rvert [a,b]}=\norm{\mu_M}([a,b])=\norm{\mu_M}_c([a,b]).
    \end{align}
    Moreover, for $\phi\colon [a,b]\to \R^d$ and $\psi\colon [a,b]\to \C$ bounded Borel-measurable functions, we have
    \[\left\lvert\int_a^b dM(\theta)\phi(\theta)\right\rvert \leq \norm{\phi}_{\infty} \Var (M)_{\rvert [a,b]}\text{ and } \left\lVert\int_a^b \psi(\theta) dM(\theta)\right\rVert \leq \norm{\psi}_{\infty} \Var (M)_{\rvert [a,b]}.\]
\end{proposition}

\begin{proof}
    Given $a=t_1<t_2<\dotsb<t_{k+1}=b$, we have
    \[M(t_{i+1})-M(t_{i})=\begin{cases}M(t_{2})=\mu_M([t_1,t_2]) & \text{if } i=1,
    \\
\mu_M([a,t_{i+1}])-\mu_M([a,t_i])=\mu_M((t_i,t_{i+1}]) & \text{if } i\ge 2.
\end{cases}\]
So the sets defined by $E_1=[t_1,t_2]$ and $E_i=(t_i,t_{i+1}], i\in \{2, \dotsc,k\}$ are disjoint, belong to $\mathcal{B}_{[a,b]}$, and $\bigcup_{i=1}^k E_i=[a,b]$. Hence $\Var(M)_{\rvert [a,b]}\leq \norm{\mu_M}([a,b])$.

To prove the converse inequality, we define $T_M(x)=\Var(M)_{\rvert [a,x]}$. We claim that $\norm{M(s)-M(t)}\leq T_M(s)-T_M(t)$, for every $s,t\in [a,b]$ with $s\geq t$. This is clearly true for $t = a$ or $t = s$ and, for $t > a$ and $s > t$, it is not difficult to see that
\begin{align*}
    T_M(s) & =\sup\left\{\sum_{i=1}^k \norm{M(t_{i+1})-M(t_{i})}: k\in \N \text{ and }a=t_1<t_2<\dotsb<t_{k+1}=s \right\}\\
    & = \sup\left\{\sum_{i=1}^k \norm{M(t_{i+1})-M(t_{i})}: k\in \N \text{ and }a=t_1<\dotsb<t_{j+1}=t<\dotsb<t_{k+1}=s \right\}.
\end{align*}
From this, it follows that
\begin{align*}
    T_M(s) & \ge \sup\left\{\sum_{i=1}^{j} \norm{M(t_{i+1})-M(t_{i})}: j\in \N \text{ and }a=t_1<\dotsb<t_{j+1}=t \right\}+ \norm{M(s)-M(t)}\\  & = T_M(t)+\norm{M(s)-M(t)},
\end{align*}
as desired. Using the continuity of $\norm{\cdot}$, we obtain, for $s\geq t$,
\[\norm {M(\hat{s})-M(\hat{t})} \leq T_M(\hat{s})-T_M(\hat{t})\]
where $\hat{s}\in \{s^-,s^+,s\}$ and $\hat{t}\in \{t^-,t^+,t\}$. Let $I$ be an interval contained in $[a,b]$ with endpoints $s\geq t$. Using the continuity from below and above of $\mu_M$ and $\mu_{T_M}$, we obtain, after some calculations, that
\[\mu_M(I)=M(\hat{s}(I))-M(\hat{t}(I)) \text{ and } \mu_{T_M}(I)=T_M(\hat{s}(I))-T_M(\hat{t}(I)),\]
where
\[\hat t(I)=\begin{cases}
t^-,& \text{if}\;t\ne a \text{ and }t\in I,\\
t^+,& \text{if}\;t=a \text{ and }t\notin I,\\
t,& \text{otherwise,} \\
\end{cases} \qquad \hat s(I)=\begin{cases}
s,& \text{if}\; s\in I\text{ and }I\ne [a,a],\\
s^-,& \text{if}\;s\notin I \text{ and }I\ne [a,a],\\
s^+& \text{if}\;I=[a,a].\\
\end{cases}\]
Therefore
\begin{equation}
\label{eq:mu_M-leq-mu_TM}
\norm{\mu_M(I)} \leq \mu_{T_M}(I),
\end{equation}
for every interval $I$ contained in $[a,b]$. The set $\mathcal A$ of all finite disjoint unions of intervals of $[a,b]$ is an algebra over $[a,b]$ and, thanks to \eqref{eq:mu_M-leq-mu_TM}, any element $E$ of $\mathcal A$ satisfies $\norm{\mu_M(E)} \leq \mu_{T_M}(E)$. If we set $\mathcal{D}=\{E\in \mathcal{B}_{[a,b]}: \norm{\mu_M(E)}\leq \mu_{T_M}(E)\}$, then $\mathcal{A}\subset \mathcal{D}$. Given $\{A_j\}_{j=1}^{+\infty} \subset \mathcal{D}$ with $A_1\subset A_2\subset \dotsb$, using the continuity from below of $\mu_M$ and $\mu_{T_M}$, we have
\begin{gather*}
    \left\lVert\mu_M\Bigg(\bigcup_{j=1}^{+\infty} A_j\Bigg)\right\rVert=\left\lVert\lim_{j\to +\infty}\mu_M (A_j)\right\rVert=\lim_{j\to +\infty}\left\lVert\mu_M (A_j)\right\rVert\leq \lim_{j\to +\infty}\mu_{T_M} (A_j)=\mu_{T_M}\Bigg(\bigcup_{j=1}^{+\infty} A_j\Bigg).
\end{gather*}
Similarly, given $\{B_j\}_{j=1}^{+\infty}\subset\mathcal D$ with $B_1\supset B_2\supset \dotsb$, we have
\[\left\lVert\mu_M\Bigg(\bigcap_{j=1}^{+\infty} B_j\Bigg)\right\rVert\leq \mu_{T_M}\Bigg(\bigcap_{j=1}^{+\infty} B_j\Bigg).\]
Therefore $\mathcal{D}$ is a monotone class containing $\mathcal{A}$ and, from the monotone class theorem for sets, the smallest monotone class containing $\mathcal{A}$ is precisely the $\sigma$-algebra generated by $\mathcal{A}$, which is exactly $\mathcal{B}_{[a,b]}$. Hence $\mathcal{D}=\mathcal{B}_{[a,b]}$. Finally,
\begin{align*}
    \norm{\mu_M}([a,b]) &=\sup\left\{\sum_{i=1}^k \norm{\mu_M(E_i)}: k\in \N, E_1, \dotsc,E_{k}\in \mathcal{B}_{[a,b]} \text{ are disjoint, } \bigcup_{j=1}^{k}E_j =[a,b]\right\}\\
    & \leq \sup\left\{\sum_{i=1}^k \mu_{T_M}(E_i): k\in \N, E_1, \dotsc,E_{k}\in \mathcal{B}_{[a,b]} \text{ are disjoint, } \bigcup_{j=1}^{k}E_j =[a,b]\right\}\\
    & =\sup\left\{ \mu_{T_M}\left(\bigcup_{i=1}^k E_i\right): k\in \N, E_1, \dotsc,E_{k}\in \mathcal{B}_{[a,b]} \text{ are disjoint, } \bigcup_{j=1}^{k}E_j =[a,b]\right\}\\
    & =\mu_{T_M}([a,b])=T_M(b)=\Var(M)_{\rvert [a,b]},
\end{align*}
yielding the first equality in \eqref{eq:variations-coincide}. 

It is clear that $\norm{\mu_M}([a,b])\leq \norm{\mu_M}_c([a,b])$. Given a family $\{E_i\}_{i \in \N}$ of disjoint sets in $\mathcal B_{[a, b]}$ such that $\bigcup_{i=1}^\infty E_i = [a,b]$, we denote $\widetilde{E_k}=\bigcup_{i=k}^{+\infty} E_i$ for every $k \in \N$. Then, for every $k \in \N$, we have
\[\sum_{i=1}^{k} \norm{\mu_M(E_i)}+\norm{\mu_M(\widetilde E_{k+1})}\leq \norm{\mu_M}([a,b]).\]
From the continuity from below of $\mu_M$, we have $\lim_{k\to +\infty} \norm{\mu_M(\widetilde E_{k+1})} = 0$, so, taking the limit as $k \to +\infty$ in the above expression, we get
\[\sum_{i=1}^{+\infty} \norm{\mu_M(E_i)}\leq \norm{\mu_M}([a,b]).\]
Since the family $\{E_i\}_{i \in \N}$ of disjoint sets in $\mathcal B_{[a, b]}$ such that $\bigcup_{i=1}^\infty E_i = [a,b]$ is arbitrary, we deduce that $\norm{\mu_M}_c([a,b])\leq \norm{\mu_M}([a,b])$, which gives the second equality in \eqref{eq:variations-coincide}. 

Proceeding now as in \cite[Section~4.1]{Cohn2013}, let $\phi\colon[a, b] \to \R^d$ be a simple function. Writing $\phi=\sum_{i=1}^k \phi_i\mathbbm{1}_{E_i}$ for some $E_1, \dotsc, E_k$ in $\mathcal B_{[a, b]}$ and $\phi_1, \dotsc, \phi_k$ in $\mathbb R^d$, we have
\[\left\lvert\int_a^b dM(\theta)\phi(\theta)\right\rvert=\left\lvert \sum_{i=1}^k\mu_M(E_i)\phi_i\right\rvert\leq \sum_{i=1}^k \norm{\mu_M(E_i)}\abs{\phi_i}\leq \norm{\phi}_\infty \Var (M)_{\rvert [a,b]}.\]
Since every bounded function is a limit of simple functions in the uniform norm, we obtain desired result. The second inequality is analogous.
\end{proof}

We now provide the explicit expression of the total variation of a matrix-valued function $M \in \mathcal N$, i.e., without a singular part.

\begin{proposition}\label{prop:tvnosingular}
Suppose that $M\in\mathcal{N}$ is given by 
\[M(\theta)=\sum_{k=1}^{+\infty} A_k \mathbbm{1}_{[-\tau_k,0]}(\theta)+N(\theta),\]
where $A_k\in \mathcal M_d (\R)$, the $\tau_k$'s, $k \in \N$, are two by two distinct and belong to $(0,1)$, and $N$ is absolute continuous. Then $\sum_{k=1}^{+\infty} \norm{A_k}<+\infty$ and
\begin{equation}\label{eq:tvequality}
    \Var (M)_{\rvert[-1,0]} = \sum_{k=1}^{+\infty} \norm{A_k} + \int_{-1}^0 \norm{N'(\theta)} d\theta.
\end{equation}
\end{proposition}

\begin{proof}
Let $\widetilde N (\theta)=M(\theta)-N(\theta)$. Then $\Var \widetilde N\leq \Var M + \Var N < +\infty$. The measure $\mu_{\widetilde N}$ associated with $\widetilde N$ in the sense of Remark~\ref{remk:M-mu} is given by
\[\mu_{\widetilde N}= \sum_{k=1}^{+\infty} A_k \delta_{-\tau_k}.\]
If we set $E_k=\{-\tau_k\}$ for $k\ge 2$ and $E_1=[-1,0] \setminus \bigcup_{k=2}^{+\infty} E_k$, then, from Proposition~\ref{proposition:tvcoincides},
\[\sum_{k=1}^{+\infty} \norm{\mu_{\widetilde N} (E_k) } = \sum_{k=1}^{+\infty} \norm{A_k} \leq \Var \widetilde N < +\infty.\]

Let us prove first that, for $m\geq 1$, if $M_m$ is the function defined by
\[
M_m(\theta)=\sum_{k=1}^m A_k\,\mathbbm{1}_{[-\tau_k,0]}(\theta)+N(\theta),\quad \theta\in[-1,0],
\]
then
\begin{equation}\label{eq:tvequalitytruncated}
\Var(M_m)_{\rvert [-1,0]}  = \sum_{k=1}^m \norm{A_k} + \int_{-1}^0 \norm{N'(\theta)} d\theta.
\end{equation}
Indeed, up to reordering $\tau_1, \dotsc, \tau_m$, we assume that $-1< -\tau_1<\dotsb< -\tau_m<0$ and we set by convention $\tau_0 = 1$. We have
\begin{equation}\label{eq:tvsumequalsumtv}
\Var(M_m)_{\rvert [-1,0]} =  \sum_{k=0}^{m-1} \Var (M_m)_{\rvert [-\tau_k,-\tau_{k+1}]}+\Var (M_m)_{\rvert [-\tau_m,0]}.
\end{equation}
We first show that, for $0\leq k\leq m-1$,
\[
\Var (M_m)_{\rvert [-\tau_k,-\tau_{k+1}]}=\norm{A_{k+1}}+\Var (N)_{\rvert [-\tau_k,-\tau_{k+1}]}.
\]
Given $\epsilon>0$, by continuity of $N$ there exists $\delta>0$ such that, for every
$t\in(-\tau_{k+1}-\delta,-\tau_{k+1})$ we have $\norm{N(-\tau_{k+1})-N(t)}<\epsilon/4$.
Take a partition $-\tau_{k}=t_1<\dotsb<t_{n+1}=-\tau_{k+1}$ of $[-\tau_k,-\tau_{k+1}]$ such that
$\sum_{j=1}^n \norm{N(t_{j+1})-N(t_{j})} \geq \Var (N)_{\rvert [-\tau_k,-\tau_{k+1}]} -\epsilon/2$.
Without loss of generality we may assume $t_n\in(-\tau_{k+1}-\delta,-\tau_{k+1})$. Then
\begin{align*}
\sum_{j=1}^n \norm{M_m(t_{j+1})-M_m(t_j)} &= \sum_{j=1}^{n-1}\norm{N(t_{j+1})-N(t_j)}+\norm{A_{k+1}+N(t_{n+1})-N(t_n)} \\
&\geq\sum_{j=1}^{n-1} \norm{N(t_{j+1})-N(t_j)} + \norm{A_{k+1}} -  \norm{N(t_{n+1})-N(t_n)}\\
&=\sum_{j=1}^{n} \norm{N(t_{j+1})-N(t_j)} + \norm{A_{k+1}} -  2\norm{N(-\tau_{k+1})-N(t_n)}\\
&\geq \Var(N)_{\rvert [-\tau_k,-\tau_{k+1}]} -\frac{\epsilon}{2} + \norm{A_{k+1}} -2\frac{\epsilon}{4} \\
&= \Var(N)_{\rvert [-\tau_k,-\tau_{k+1}]}  + \norm{A_{k+1}} -\epsilon.
\end{align*}
Thus
\[\Var(M_m)_{\rvert [-\tau_k,-\tau_{k+1}]}  \geq  \Var (N)_{\rvert [-\tau_k,-\tau_{k+1}]}  + \norm{A_{k+1}} -\epsilon,\]
for every $\epsilon>0$, which implies
\[\Var (M_m)_{\rvert [-\tau_k,-\tau_{k+1}]} \geq  \Var (N)_{\rvert [-\tau_k,-\tau_{k+1}]}  + \norm{A_{k+1}},\quad 0\leq k \leq m-1.\]
The reverse inequality is clear, hence we have
\begin{equation}\label{eq:prop:tvnosingula1}
    \Var (M_m)_{\rvert [-\tau_k,-\tau_{k+1}]} =  \Var (N)_{\rvert [-\tau_k,-\tau_{k+1}]} + \norm{A_{k+1}},\quad 0\leq k \leq m-1.
\end{equation}
On the other hand, since $M_m =\sum_{k=1}^m A_k + N$ on $[-\tau_m, 0]$, we have $\Var(M_m)_{\rvert [-\tau_m, 0]} = \Var(N)_{\rvert [-\tau_m, 0]}$. Summing \eqref{eq:prop:tvnosingula1} over $k$ in $\{0, \dotsc, m-1\}$ and using the previous fact, we obtain the equality \eqref{eq:tvequalitytruncated} after recalling that $\Var (N)_{\rvert [-1,0]} = \int_{-1}^0 \norm{N'(\theta)}d\theta$ since $N$ is absolutely continuous.

To conclude, note that
\[\Var {(M-M_m)}=\Var {\Bigg(\sum_{k=m+1}^{+\infty} A_k\,\mathbbm{1}_{[-\tau_k,0]}\Bigg)} \leq \sum_{k=m+1}^\infty \norm{A_k}\to 0\]
as $m\to +\infty$. Since
$\abs{\Var(M)_{\rvert [-1,0]}-\Var(M_m)_{\rvert [-1,0]}}\leq \Var(M-M_m)_{\rvert [-1,0]}$, we have
\[\Var(M)_{\rvert [-1,0]}=\lim_{m\to +\infty}\Var(M_m)_{\rvert [-1,0]}=\lim_{m\to +\infty} \Big(\sum_{k=1}^m \norm{A_k} + \int_{-1}^0 \norm{N'(\theta)}\,d\theta\Big),\]
yielding \eqref{eq:tvequality}.
\end{proof}

Our next two results concern the uniqueness of weak-$*$ $L^\infty$ limits with respect to different measures. The first result, Lemma~\ref{lemma:measure1} below, is an intermediate step to get Proposition~\ref{prop:measure2}. To state and prove them, we use the standard notation $\mu \ll \nu$ to say that the measure $\mu$ is absolutely continuous with respect to the measure $\nu$, and $\mu \perp \nu$ to denote that the measures $\mu$ and $\nu$ are mutually singular.

\begin{lemma}\label{lemma:measure1}
        Let $\mathcal X = (X, \mathfrak S)$ be a measurable space, $\nu$ and $\mu$ be finite and positive measures on $\mathcal X$, and $(\eta_n)_{n\in\N}$ be a sequence of complex-valued measurable functions defined on $X$ such that
        \[\eta_n\xrightharpoonup{*} f \text{ in } L^\infty(\mathcal X,\C,\nu)\quad\text{ and }\quad\eta_n\xrightharpoonup{*} g \text{ in } L^\infty(\mathcal X,\C,\mu)  \]
        for some measurable functions $f$ and $g$. Suppose in addition that there exist positive measures $\mu_{ac}$ and $\bar\mu$ such that $\mu=\mu_{ac}+\bar\mu$, $\mu_{ac}\ll\nu$, and $\mu_{ac} \perp \bar\mu$. Then $f=g$ $\mu_{ac}$-a.e.
    \end{lemma}
\begin{proof}
        By hypothesis, we have
        \[\int_X\eta_n\varphi d\mu \longrightarrow \int_X g\varphi d\mu,\quad \forall \varphi\in L^1(\mathcal X,\C,\mu),\]
        or, equivalently,
        \begin{equation}
        \label{eq:convergence}
        \int_X \eta_n\varphi d\mu_{ac}+\int_X \eta_n\varphi d\bar\mu  \longrightarrow \int_X g\varphi d\mu_{ac}+\int_X g\varphi d\bar\mu,
        \end{equation}
        for every $\varphi\in L^1(\mathcal X,\C,\mu)$. Since $\mu_{ac} \perp \bar\mu$, there exist $A_1$ and $A_2$ in $\mathfrak S$ such that $A_1 \cap A_2 = \emptyset$, $A_1 \cup A_2 = X$, and $\mu_{ac}(A_1) = \bar\mu(A_2) = 0$.
        For any $\varphi\in L^1(\mathcal X,\C,\mu_{ac})$ it is clear that $\widetilde\varphi := \varphi\mathbbm{1}_{A_2}$ belongs to $L^1(\mathcal X,\C,\mu)$. So, from \eqref{eq:convergence}, we obtain, for every $\varphi\in L^1(\mathcal X,\C,\mu_{ac})$,
        \[\int_X \eta_n\varphi d\mu_{ac}=\int_X \eta_n\widetilde\varphi d\mu_{ac}+ \underbrace{\int_X \eta_n\widetilde\varphi d\bar\mu}_{=0}   \longrightarrow \int_X g\widetilde\varphi d\mu_{ac}+\underbrace{\int_X g\widetilde\varphi d\bar\mu}_{=0}=\int_X g\varphi d\mu_{ac}.\]
        Thus one has $\eta_n\xrightharpoonup{*}g$ in $L^\infty(\mathcal X,\C,\mu_{ac})$. 

        On the other hand, since $\mu_{ac}\ll \nu $ there exists $\psi=\frac{d\mu_{ac}}{d\nu}\in L^1(\mathcal X,\C,\nu)$. From the convergence
        \[ 
         \int_X \eta_n\varphi d\mu_{ac}\longrightarrow \int_X g\varphi d\mu_{ac},\quad\forall \varphi\in L^1(\mathcal X,\C,\mu_{ac})
        \]
        we have in particular
        \[\int_X \eta_n\varphi \psi d\nu\longrightarrow \int_X g\varphi\psi d\nu,\quad\forall \varphi\in \{\mathbbm{1}_A: A\in \mathfrak S\}.\]
        Since $\varphi\psi \in L^1(\mathcal X,\C,\nu)$ for every $\varphi\in \{\mathbbm{1}_A: A\in \mathfrak S \}$ and $\eta_n\xrightharpoonup{*} f \text{ in } L^\infty(\mathcal X,\C,\nu)$, we obtain
        \[\int_X \eta_n\varphi \psi d\nu \longrightarrow \int_X f\varphi\psi d\nu,\]
        so, from the uniqueness of the limit,
        \[\int_X g\varphi\psi d\nu=\int_X f\varphi\psi d\nu, \quad\forall \varphi\in \{\mathbbm{1}_A: A\in \mathfrak S \}. \]
        Therefore,
        \begin{equation}\label{eq:lemma:measure1}
            \int_X  (f-g) \varphi \psi d\nu=\int_X  (f - g) \varphi d\mu_{ac} = 0, \quad \forall \varphi\in \{\mathbbm{1}_A: A\in \mathfrak S \}.
        \end{equation}
        We claim that $f=g$ $\mu_{ac}$-a.e. Indeed, we set $D=\{\xi \in X:f(\xi)\ne g(\xi)\}$ and suppose by contradiction that $\mu_{ac}(D)>0$. Then we have $\mu_{ac}(D_{\Re})>0$ or $\mu_{ac}(D_{\Im})>0$, where $D_{\Re}=\{\xi \in X :\Re f(\xi)\ne \Re g(\xi)\}$ and $D_{\Im}=\{\xi \in X :\Im f(\xi)\ne \Im g(\xi)\}$. Without loss of generality, suppose that $\mu_{ac}(D_{\Re})>0$, so $\mu_{ac}(D_{\Re}^+)>0$ or $\mu_{ac}(D_{\Re}^-)>0$, where $D_{\Re}^+=\{\xi \in X :\Re f(\xi)>\Re g(\xi)\}$ and $D_{\Re}^-=\{\xi \in X :\Re f(\xi)<\Re g(\xi)\}$. If we assume $\mu_{ac}(D_{\Re}^+)>0$, then
        \[\int_X (\Re f - \Re g)\mathbbm{1}_{D_{\Re}^+} d\mu_{ac}>0\]
        which contradicts \eqref{eq:lemma:measure1}. The same contradiction happens if we assume $\mu_{ac}(D_{\Re}^-) > 0$, and therefore, the assertion holds.
    \end{proof}
    
We now generalize Lemma~\ref{lemma:measure1} to the case of an arbitrary number of measures.
    
\begin{proposition}\label{prop:measure2}
Let $\mathcal X = (X, \mathfrak S)$ be a measurable space, $\mu_1,\dotsc,\mu_K$  be finite and positive measures on $\mathcal X$, and $(\eta_n)_{n\in\N}$ be a sequence of complex-valued measurable functions defined on $X$. Suppose that, for every $k\in\{1,\dotsc,K\}$, there exists a measurable function $\bar{\eta}_k$ such that
\[\eta_n\xrightharpoonup{*} \bar{\eta}_k\text{ in } L^\infty(\mathcal X,\C,\mu_k).\]
Then there exists a measurable function $\bar{\eta}$ such that $\bar{\eta}=\bar{\eta}_k$ $\mu_k$-a.e., for every $k\in\{1, \dotsc,K\}$.
\end{proposition}

\begin{proof}
       Let us prove this by induction in $K$. For $K=1$, we just take $\bar{\eta}=\bar{\eta}_1$. Suppose that the result holds for the family $\mu_1, \dotsc, \mu_N$ for some $N \in \{1, \dotsc, K - 1\}$. Thus, by the induction hypothesis, there exists a measurable function $\widetilde{\eta}$ such that 
       \[\eta_n\xrightharpoonup{*} \widetilde{\eta}\text{ in } L^\infty(\mathcal X,\C,\mu_k), \quad \forall k\in\{1, \dotsc,N\} \]
       and 
       \[\eta_n\xrightharpoonup{*} \bar{\eta}_{N+1}\text{ in } L^\infty(\mathcal X,\C,\mu_{N+1}).\]
    Now, we decompose $\mu_{N+1}$ using the Radon--Nikodym--Lebesgue decomposition as follows:
    \begin{enumerate}
        \item We first write $\mu_{N+1} = \mu_{ac}^N + \mu_s^N$, where $\mu_{ac}^N\ll \mu_N$ and $\mu_{s}^N \perp \mu_N$;
        \item Recursively, for $j=N, \dotsc,2$, we decompose $\mu_s^j=\mu_{ac}^{j-1}+\mu_s^{j-1}$, where $\mu_{ac}^{j-1}\ll \mu_{j-1}$ and $\mu_s^{j-1}\perp \mu_{j-1}$.
    \end{enumerate}
    Thus, we have
    \[\mu_{N+1}=\mu_{ac}^N+\mu_{ac}^{N-1}+ \dotsb +\mu_{ac}^1+\mu_s^1.\]
    Now, for $k=1, \dotsc,N$ we have $\mu_s^k\perp\mu_k$, so we take $\widetilde{A}_k, \widetilde{B}_k\subset X$ such that $\widetilde{A}_k\cap\widetilde{B}_k=\emptyset$, $\widetilde{A}_k\cup \widetilde{B}_k=X$, and $\mu_{s}^k(\widetilde{A}_k)=0$, $\mu_k(\widetilde{B}_k)=0$. We define the sets $A_j,B_j$ for $j \in \{1, \dotsc,N\}$ by $A_j:=\widetilde{A}_j\cup\dotsb \cup \widetilde{A}_N $ and $B_j:=\widetilde{B}_j\cap\dotsb \cap \widetilde{B}_N$, and by construction we have $A_j \cap B_j = \emptyset$ and $A_j \cup B_j = X$.
    Consider $\bar{\eta}$ given by
    \[\bar{\eta}(\xi)=\begin{cases}
\widetilde{\eta}(\xi),& \text{if } \xi\in A_1, \\ 
\bar{\eta}_{N+1}(\xi),& \text{if } \xi\in B_1.
\end{cases}\]
To conclude the proof, we will show that $\bar{\eta}=\bar{\eta}_k$ $\mu_k$-a.e.\ for every $k\in\{1, \dotsc,N+1\}$. For every $k \in \{1, \dotsc, N\}$, note that the set $B=\{\xi \in X: \bar{\eta}(\xi)\ne \widetilde{\eta}(\xi)\}$ is a subset of $B_1$, thus $\mu_k(B)\leq \mu_k(B_1)\leq \mu_k(\widetilde{B}_k)=0$. Consequently, $\bar{\eta} = \widetilde{\eta} = \bar\eta_k$ $\mu_k$-a.e.\ for every $k \in \{1, \dotsc, N\}$. We are thus left to show that $\bar\eta = \bar\eta_{N + 1}$ $\mu_{N+1}$-a.e.

Since $\mu_s^k=\mu_{ac}^{k-1}+ \dotsb + \mu_{ac}^1+\mu_s^1$ for $k \in \{1, \dotsc, N\}$ and all those measures are positive, then $\mu_s^k(\widetilde{A}_k)=0$ implies $\mu_s^1(\widetilde{A}_k)=0$ for every $k\in\{1, \dotsc,N\}$. So $\mu_s^1(A_1)=0$. 

For every $k\in\{1, \dotsc,N\}$, we write $\mu_{N+1}=\bar \mu^k+\mu_{ac}^k$, where $\bar \mu^k=\mu_{ac}^N+\dotsb+\mu_{ac}^{k+1}+\mu_{s}^k$, and we set $A^k = \widetilde B_k \cup \widetilde A_{k+1}\cup \dotsb \cup \widetilde A_N$ and $B^k = X \setminus A^k = \widetilde A_k \cap \widetilde B_{k+1} \cap \dotsb \cap \widetilde B_N$. Since $\mu_s^j = \mu_{ac}^{j-1}+ \dotsb+\mu_{ac}^1+\mu_s^1$ for $j \in \{1, \dotsc, N\}$ and all those measures are positive, we deduce that, if $j \in \{k + 1, \dotsc, N\}$, then $\mu_{ac}^k(\widetilde A_j) \leq \mu_s^j(\widetilde A_j) = 0$. In addition, since $\mu_{ac}^k \ll \mu_k$ and $\mu_k(\widetilde B_k) = 0$, we have $\mu_{ac}^k(\widetilde B_k) = 0$. Those two facts show that $\mu_{ac}^k(A^k) = 0$. In addition, for $j \in \{k + 1, \dotsc, N\}$, we have from the definition of $B^k$ that $\mu_{ac}^j(B^k) \leq \mu_{ac}^j(\widetilde B_j) = 0$ and $\mu_{s}^k(B^k) \leq \mu_s^k(\widetilde A_k) = 0$, showing that $\bar \mu^k (B^k) = 0$. Hence $\mu_{ac}^k \perp \bar\mu^k$ and, applying Lemma~\ref{lemma:measure1} with $f = \widetilde\eta$, $g = \bar\eta_{N + 1}$, $\nu = \mu_k$, and $\mu = \mu_{N+1} = \mu_{ac}^k + \bar\mu^k$, we have that $\widetilde{\eta}=\bar{\eta}_{N+1}$ $\mu_{ac}^k$-a.e.\ for every $k \in \{1, \dotsc, N\}$. Thus, for $A=\{\xi \in X: \widetilde{\eta}(\xi)\ne \bar{\eta}_{N+1}(\xi)\}$, one has $\mu_{ac}^k(A)=0$, for every $k=1,\dotsc,N$. Since
\[A^* := \{\xi\in X: \bar{\eta}(\xi)\ne \bar{\eta}_{N+1}(\xi)\}=A_1\cap A,\]
we obtain
\[\mu_{N+1}(A^*)\leq \mu_{ac}^N(A)+ \dotsb+\mu_{ac}^1(A)+\mu_s^1(A_1)=0.\]
Therefore $\bar\eta=\bar\eta_{N+1}$ $\mu_{N+1}$-a.e., as required.
\end{proof}   
   
\subsection{Other technical results}
\label{app:other}

We next provide an elementary argument for the exponential bound \eqref{eq:exp-estimate} for solutions of \eqref{sys:general}. Recall that such a bound can also be obtained from standard functional-analytic tools, as detailed in \cite[Section~12.3]{Hale1977Theory}.


\begin{proof}[Proof of \eqref{eq:exp-estimate}]
    Notice that we may assume $\Var(M)_{\rvert [s,0]} \to 0$ as $s \to 0^-$; otherwise, define
    \[
        \widetilde{M}(\theta) = 
        \begin{cases}
            (\I - A)^{-1} M(\theta), & \text{if } -1 \leq \theta < 0,\\
            (\I - A)^{-1} M(0^-), & \text{if } \theta = 0,
        \end{cases}
    \]
    with $A = M(0) - M(0^-)$. Then $\Sigma_{\widetilde{M}}$ is equivalent to $\Sigma_M$ (in the sense that their trajectories coincide) and $\Var(\widetilde{M})_{\rvert [s,0]} \to 0$ as $s \to 0^-$.

   Let $x$ be the solution of \eqref{sys:general} with initial condition $\phi$. For $\eta>0$, define $y_\eta \colon [-1, +\infty) \to \mathbb R^d$ by $y_\eta(t) = e^{-\eta t} x(t)$. Then $y_\eta$ satisfies
    \[
        y_\eta(t) = \int_{-1}^0 dM_\eta(\theta)\, y_\eta(t+\theta),\quad t \geq 0,
    \]
    where
    \[
        M_\eta(\theta) = \int_{-1}^{\theta} e^{\eta \tau}\, dM(\tau), \quad \theta\in [-1,0].
    \]
    From the definition of $M_\eta$ and Proposition~\ref{proposition:tvcoincides}, it follows that, for every $s \in (-1,0)$,
    \[
        \Var(M_\eta)_{\rvert[-1,0]} 
        = \Var(M_\eta)_{\rvert[-1,s]} + \Var(M_\eta)_{\rvert[s,0]} 
        \le e^{\eta s}\Var(M)_{\rvert[-1,s]} + \Var(M)_{\rvert[s,0]}.
    \]
    Since $\Var(M)_{\rvert[s,0]} \to 0$ as $s \to 0^-$, fix $\bar{s} \in (-1, 0)$ so that $\Var(M)_{\rvert[\bar{s},0]}<1/2$. Then, for 
    $\eta$ sufficiently large,
    $e^{\eta \bar{s}}\Var(M)_{\rvert[-1,\bar{s}]}<1/2$ and hence $\Var(M_\eta)_{\rvert [-1,0]} < 1$. Applying Proposition~\ref{proposition:tvcoincides} once more, we have, for all $t \geq 0$,
    \[
        \abs{y_\eta(t)} 
        = \left\lVert \int_{-1}^0 dM_\eta(\theta)\, y_\eta(t+\theta) \right\rVert
        \le \Var(M_\eta)_{\rvert[-1,0]} \lVert (y_\eta)_t \rVert_\infty.
    \]
    By \cite[Lemma~1]{Mazenc2015Trajectory}, one deduces that there exist positive constants $\widetilde{K}$ and $\beta$ such that
    \begin{equation*}
        \lVert (y_\eta)_t \rVert_\infty 
        \le \widetilde{K} e^{-\beta t} \lVert (y_\eta)_0 \rVert_\infty,
        \quad t \ge 0.
    \end{equation*}
    Since $x(t) = e^{\eta t} y_\eta(t)$, one gets the conclusion.
\end{proof}

Our next result proves that $\mathcal N\subset \mathcal L$, where the classes $\mathcal N$
and $\mathcal L$ have been defined in \eqref{eq:def-N-L-G}.

\begin{proposition}\label{prop:positiveinf}
    Suppose that $M\in\mathcal{N}$ is given by \eqref{eq:NS}, then $S_M < +\infty$ and
    \[\inf_{\Re s \ge \alpha} \abs{\Delta_M(s)} > 0\]
    for every $\alpha>S_M$.
\end{proposition}

\begin{proof}
Since $\mathcal N \subset \mathcal W$, we have $S_M \leq \omega_M < +\infty$. Suppose by contradiction that there exists $\alpha > S_M$ such that $\inf_{\Re s \ge \alpha} \abs{\Delta_M(s)}=0$. Then there exists a sequence $(s_n)_{n\in \N}$ such that $\Delta_M(s_n)\to 0$ as $n\to +\infty$ with $\Re s_n\ge \alpha$. Since $\Re s_n \ge \alpha > S_M$, it follows that $\abs{s_n} \to +\infty$, otherwise we would obtain a subsequence converging to a zero of $\Delta_M$ located in the half-plane $\{z \in \C : \Re z \geq \alpha\}$. If $(\Re s_n)_{n \in \mathbb N}$ were unbounded, then, up to a subsequence, we would have $\Re s_n\to +\infty$, which would imply $\Delta_M(s_n)\to \det(\I)=1$. Thus $(\Re s_n)_{n \in \mathbb N}$ is bounded and converges, up to a subsequence, to some $\beta\ge \alpha$. Since $\abs{s_n} \to +\infty$, we have $\abs{\Im s_n} \to +\infty$ and, up to replacing by the complex conjugate, we may assume $\Im s_n\to +\infty$. Writing $s_n=x_n+iy_n$, we obtain $x_n\to \beta$ and $y_n\to +\infty$. Define
\[
\Delta_n(s)=\Delta_M(s+iy_n)=\det \left(\I - \sum_{k=1}^{+\infty} A_k e^{-s\tau_k}e^{-iy_n\tau_k}-\int_{-1}^0 N'(\theta) e^{s\theta}e^{iy_n\theta}\,d\theta\right).
\]
Up to extracting subsequences through a diagonal procedure, we get that, for every $k\geq 1$, the sequence $(e^{-iy_n\tau_k})_{n \in \N}$ converges to $e^{i\theta_k}$ for some angle $\theta_k \in [0, 2\pi)$. By the Riemann--Lebesgue lemma, we have
\[
\int_{-1}^0 N^\prime(\theta)e^{s\theta}e^{iy_n\theta}\,d\theta \xrightarrow[n \to +\infty]{} 0,
\]
uniformly with respect $s$ belonging to any compact subset of the complex plane. As a consequence,
\[
\Delta_n(s) \xrightarrow[n \to +\infty]{} \Delta_{\ast}(s)=\det\left(\I - \sum_{k=1}^{+\infty} A_k e^{-s\tau_k}e^{i\theta_k}\right),
\]
uniformly with respect to $s$ belonging to a compact neighborhood of $\beta$. Observe that
\[\abs{\Delta_M(x_n+iy_n)-\Delta_\ast(\beta)}\leq \abs{\Delta_M(x_n+iy_n)-\Delta_M(\beta+iy_n)}+\abs{\Delta_n(\beta)-\Delta_\ast(\beta)}.\]
Since $\det(\cdot)$ is uniformly continuous on compact sets, we have that $\lim_{n\to +\infty} \abs{\Delta_M(x_n+iy_n)-\Delta_M(\beta+iy_n)} = 0$, so that $\Delta_\ast (\beta)=\lim_{n\to +\infty} \Delta_M(s_n)=0$. By~\cite[Hurwitz's Theorem, p.~231]{Gamelin2001}, for $\delta=\alpha-S_M$ there exists, for $n$ sufficiently large, some $s' \in \C$ such that $\abs{s'-\beta} < \delta$ and $\Delta_n(s')=0$. Therefore $s^*=s'+iy_n$ is a root of $\Delta_M$ satisfying $\Re s^*= \Re s' > \beta - \delta \ge \alpha - \delta = S_M$, yielding the desired contradiction. 
\end{proof}

The last technical result we present is a generalization of the maximum modulus principle to matrix-valued functions of several variables, replacing the modulus both by a norm and by the spectral radius.

\begin{lemma}\label{lemma:max_principle}
    Let $C_1, \dotsc ,C_N\subset\C$ be open sets and $f\colon C_1\times\dotsb\times C_N \to \mathcal{M}_d(\C)$ be analytic. For every $j\in \{1, \dotsc,N\}$, suppose that $D_j$ is an open, bounded and connected set such that $\cl D_j \subset C_j$, then
    \begin{equation}\label{eq:lemma:max_modulus_principle}
        \max_{z\in \cl D_1\times \dotsb\times \cl D_N} \norm{f(z)} = \max_{z\in \partial D_1\times \dotsb\times \partial D_N} \norm{f(z)},
    \end{equation}
    where $\partial D_j=\cl D_j \setminus D_j$. Moreover, 
    \begin{equation}\label{eq:lemma:max_modulus_principle1}
        \max_{z\in \cl D_1\times \dotsb\times \cl D_N}\rho(f(z)) = \max_{z\in \partial D_1\times \dotsb\times \partial D_N}  \rho(f(z)).
    \end{equation}
     
\end{lemma}
\begin{proof}
    Let $z^*=(z^*_1, \dotsc,z^*_N)\in \cl D_1\times \dotsb\times \cl D_N $ such that 
    \[
        \norm{f(z^*)}=\max_{z\in \cl D_1\times \dotsb\times \cl D_N} \norm{f(z)}.
    \]
    Let us show that there exists $\bar z= (\bar z_1, \dotsc,\bar z_N)\in \partial D_1\times \dotsb\times \partial D_N$ such that $\norm{f(z^*)}=\norm{f(\bar z)}$. For $j\in \{1, \dotsc,N\}$, if $z^*_j\in \partial D_j$ we set $\bar z_j:=z^*_j$, otherwise we define $g\colon C_j \to \mathcal{M}_d(\C)$ by
    \[g(\xi)=f(z^*_1, \dotsc,z^*_{j-1},\xi,z^*_{j+1}, \dotsc,z^*_N).\]
    So $g_{\rvert C_j}$ is analytic and, from the maximum modulus principle theorem for vector-valued functions (see \cite[Section~III.14, p.~203]{dunford1988linear}), we have that $\norm{g(\cdot)}$ is constant on $D_j$, hence for any $\bar z_j\in \partial D_j$, we have $\norm{g(\bar z_j)} = \norm{g(z^*_j)}$. If we apply this argument for any coordinate, we get the desired $\bar z= (\bar z_1, \dotsc,\bar z_N)\in \partial D_1\times \dotsb\times \partial D_N $ and this shows \eqref{eq:lemma:max_modulus_principle}.

    Suppose by contradiction that \eqref{eq:lemma:max_modulus_principle1} does not hold. We set
    \[\epsilon=\max_{z\in \cl D_1\times \dotsb\times \cl D_N}\rho(f(z)) -\max_{z\in \partial D_1\times \dotsb\times \partial D_N}  \rho(f(z))>0.\]
    From Gelfand's formula, we have  
    \[\lim_{n\to +\infty} \norm{A^n}^{\frac{1}{n}}=\rho(A),\]
    and the limit is uniform on compact sets of matrices (see, e.g., \cite[Lemma~4.1]{Colonius2019Decay}). Since $\bar D:= \cl D_1\times \dotsb \times \cl D_N$ is compact and $f$ is continuous, then $f(\bar D)$ is compact, so there exists $M\in \N$ such that
    \begin{equation}\label{eq:lemma:max_modulus_principle2}
        \abs*{\rho(f(z))-\norm{f(z)^n}^{\frac{1}{n}}} < \frac{\epsilon}{2}
    \end{equation}
    for every $n\geq M$ and $z\in \bar D$. Let $z^*\in \bar D$ be such that $\rho(f(z^*)) = \max_{z\in \bar D}\rho (f(z))$ and $z_*\in \partial D:= \partial D_1\times \dotsb\times \partial D_N$ such that $\rho(f(z_*)) = \max_{z\in \partial D}\rho (f(z))$. From \eqref{eq:lemma:max_modulus_principle2} and recalling that $\rho(f(z^*))=\rho(f(z_*))+\epsilon$ we have
    \begin{align*}
        \norm{f(z^*)^M}^\frac{1}{M}&>\rho(f(z^*))-\frac{\epsilon}{2}\\
        &=\rho(f(z_*))+\frac{\epsilon}{2}=\max_{z\in \partial D} \rho(f(z))+\frac{\epsilon}{2}\\
        &\geq \max_{z\in \partial D}\norm{f(z)^M}^\frac{1}{M}.
    \end{align*}
    Therefore $\norm{f(z^*)^M}>\max_{z\in \partial D}\norm{f(z)^M}$, which contradicts \eqref{eq:lemma:max_modulus_principle} applied to $F(z)=f(z)^M$.
\end{proof}

\bibliographystyle{abbrv}
\bibliography{main}

\end{document}